\documentclass{amsart}
\usepackage[mathscr]{eucal}% So-called Euler fonts, allows \mathscr
\usepackage{amssymb}% allows special arrows like >-> e.g.
\usepackage{pifont}
\usepackage{graphicx}
\usepackage{psfrag} % Necessary to put latex into .eps text fragments
\usepackage[usenames,dvipsnames]{color}
\usepackage[normalem]{ulem}% allows \sout to cross-out text.
\usepackage{amsthm}% allows theorem* , e.g.
\usepackage{bbm}% same as above but slightly less nice; use if \mathbb should remain as `normal'
\usepackage{enumerate}% allows easy change of label
\usepackage{array}% allows array
\usepackage{amsmath}% allows pmatrix

\usepackage{hyperref}
\hypersetup{colorlinks=true,linkcolor={Brown},citecolor={Brown},urlcolor={Brown}}% replace BrickRed by Brown?

\usepackage{cleveref}% allows references via \cref and \Cref which automatically repeat the name (Thm, Prop, etc) of the ref.

\numberwithin{equation}{section}% makes equat numb contain the section % superseded by the above
\makeindex

\usepackage[all]{xy}
\CompileMatrices

\newdir{ >}{{}*!/-10pt/\dir{>}}

\swapnumbers %pour que le numéro apparaisse devant le théorème

\newtheorem{Thm}[equation]{Theorem}

\newtheorem{Prop}[equation]{Proposition}
\newtheorem{Lem}[equation]{Lemma}
\newtheorem{Cor}[equation]{Corollary}

\theoremstyle{remark}
\newtheorem{Rem}[equation]{Remark}
\newtheorem{Def}[equation]{Definition}

\newtheorem{Not}[equation]{Notation}
\newtheorem{Exa}[equation]{Example}
\newtheorem{Exas}[equation]{Examples}

\newtheorem{Conv}[equation]{Convention}
\newtheorem{Hyp}[equation]{Hypotheses}
\newtheorem{Rec}[equation]{Recollection}

\newcommand{\nc}{\newcommand}
\nc{\dmo}{\DeclareMathOperator}

\dmo{\Ab}{Ab}
\dmo{\add}{add} % the additive hull (= closure under sums and direct summands)
\dmo{\Aut}{Aut}
\dmo{\bicMack}{\biMack_{\mathsf{ic}}} % the bicategory of i.c. Mackey 2-functors
\dmo{\biMack}{\mathsf{Mack}} % the bicategory of Mackey 2-functors
\dmo{\Ch}{Ch}% ground notation for chain complexes
\dmo{\CoInd}{CoInd}
\dmo{\Der}{D}% ground notation for derived categories
\dmo{\Db}{D^b}% ground notation for bounded derived categories
\dmo{\End}{End}
\dmo{\Fun}{\mathrm{Fun}} % the 2-category of 2-functors from enriched category theory
\dmo{\Hom}{Hom}
\dmo{\Ho}{Ho}
\dmo{\img}{im}
\dmo{\incl}{incl}
\dmo{\Ind}{Ind}
\dmo{\inj}{in} % notation for canonical injections
\dmo{\Inj}{Inj} % injective modules/objects
\dmo{\Ker}{Ker}
\dmo{\Mackey}{Mack} % the category of ordinary Mackey functors
\dmo{\Map}{Map}%
\dmo{\Mod}{Mod}% sheaves of modules
\dmo{\Qcoh}{Qcoh}% quasi-coherent sheaves over a scheme
\dmo{\coh}{coh}% coherent sheaves over a scheme
\dmo{\fgmod}{mod}
\dmo{\stmod}{stmod}
\dmo{\StMod}{StMod}
\dmo{\Mor}{Mor}%
\dmo{\Obj}{Obj}
\dmo{\Proj}{Proj} % projective modules/objects
\dmo{\fgproj}{proj} % fg projective modules/objects
\dmo{\pr}{pr}
\dmo{\PsFunJJ}{\PsFun_{\JJ_!}^{\JJ^\prime\textrm{\!-}\mathsf{oplax}}}
\dmo{\PsFunJop}{\PsFun_{{{\JJ}_{{}_{*}}}}}
\dmo{\PsFunJ}{\PsFun_{\JJ_!}}
\dmo{\PsFunoplax}{\PsFun^{\mathsf{oplax}}}
\dmo{\PsFun}{\mathsf{PsFun}} % the bicategory of pseudo-functors
\dmo{\Rad}{Rad}
\dmo{\Res}{Res}
\dmo{\SH}{SH}% ground name for cat of spectra
\dmo{\Sh}{Sh}
\dmo{\Spanname}{{\sf Span}}
\dmo{\Spec}{Spec}
\dmo{\Stab}{Stab}% stable category of non-fin. gen. mod.
\dmo{\twoFun}{2\mathsf{Fun}}
\dmo{\tr}{tr}

\nc{\ababs}{{\sl ab absurdo}}
\nc{\Add}{\mathsf{Add}}
\nc{\ADD}{\mathsf{ADD}}
\nc{\ADDic}{\mathsf{ADD}_{\mathsf{ic}}}
\nc{\adhoc}{{\sl ad hoc}}
\nc{\adjto}{\rightleftarrows}
\nc{\adj}{\dashv\,}
\nc{\afortiori}{{\sl a fortiori}}
\nc{\aka}{{a.\,k.\,a.}\ }
\nc{\all}{\mathsf{all}}% all 1-cells.
\nc{\apriori}{{\sl a priori}}
\nc{\ass}{\mathrm{ass}} % associator
\nc{\bbA}{\mathbb{A}}
\nc{\bbB}{\mathbb{B}}
\nc{\bbC}{\mathbb{C}}
\nc{\bbD}{\mathbb{D}}
\nc{\bbF}{\mathbb{F}}
\nc{\bbI}{\mathbb{I}}
\nc{\bbM}{\mathbb{M}}
\nc{\bbN}{\mathbb{N}}
\nc{\bbP}{\mathbb{P}}
\nc{\bbQ}{\mathbb{Q}}
\nc{\bbR}{\mathbb{R}}
\nc{\bbZ}{\mathbb{Z}}
\nc{\bs}{\backslash}
\nc{\BurnG}{\cat{A}(G)}
\nc{\cat}[1]{\mathcal{#1}}
\nc{\Cat}{\mathsf{Cat}}
\nc{\CAT}{\mathsf{CAT}}
\nc{\cf}{{\sl cf.}\ }
\nc{\Cf}{{\sl Cf.}\ }
\nc{\colim}{\mathop{\mathrm{colim}}}
\nc{\costar}{**}% for the (-)^* embedding with \beta^{-1} on 2-cells
\nc{\co}{{\mathrm{co}}}
\nc{\DD}{\cat{D}}% a derivator
\nc{\Displ}{\displaystyle}
\nc{\diag}[1]{\overline{#1}} % (essential) diagonal-part operation
\nc{\offdiag}[1]{{#1}^\dagger} % off-diagonal-part operation
\nc{\doublequot}[3]{#1\backslash #2/#3}% double-cosets
\nc{\Ecell}{\rotatebox[origin=c]{90}{$\Downarrow$}} % treated as others too! But not to be used in-line!!! [I]
\nc{\eg}{{\sl e.g.}\ } % made similar to the previous two
\nc{\Eg}{{\sl E.g.}\ } % we needed a capital one!
\nc{\eps}{\varepsilon}
\nc{\equalby}[1]{\overset{\textrm{#1}}{=}}
\nc{\exact}{\mathsf{ex}}
\nc{\faithful}{\mathsf{faithful}}% faithful
\nc{\faith}{\mathsf{faithf}}
\nc{\final}{\textrm{\scriptsize{\ding{93}}}} % modified asterisk
\nc{\Funadd}{\Fun_{\amalg}}% additive functors, in the sense of Mack 1
\nc{\Funplus}{\Fun_{+}}% additive functors. Can be tweaked...
\nc{\fun}{\mathrm{fun}} % "funtorator" of pseudo-functors (or whatever it's called)
\nc{\GG}{\mathbb{G}}% the 2-category of finite groupoids `of interest'
\nc{\gpdG}{{\groupoidf_{\!\smallslash\!G}}} % the "correct" 2-category of groupoids for Mackey functors for G, i.e. the comma category of faithful functors to G
\nc{\gpdGzero}{{\groupoidf_{\!\smallslash\!G_0}}\!} % variation with G_0
\nc{\gpdfover}[1]{\groupoidf_{\!\smallslash\!#1}}
\nc{\gpd}{\groupoid}%
\nc{\gps}{\mathsf{groups}} % category of finite groups
\nc{\groconn}{\groupoid_{\mathsf{conn}}}% connected finite groupoids (auxiliary)
\nc{\groupoidf}{\groupoid{}^{\smallfaithful}}% finite groupoids with faithful morphisms
\nc{\groupoid}{\mathsf{gpd}}% finite groupoids
\nc{\group}{\mathsf{group}} % category of finite groups, variant...
\nc{\Gsets}{G\sset}
\nc{\HGfK}{\doublequot{H}{G}{f(K)}}%
\nc{\HGK}{\doublequot HGK}% most used
\nc{\Homcat}[1]{\Hom_{\cat #1}}
\nc{\hooklongleftarrow}{\longleftarrow\joinrel\rhook}
\nc{\hooklongrightarrow}{\lhook\joinrel\longrightarrow}
\nc{\hook}{\hookrightarrow}
\nc{\Hsets}{H\mathsf{-sets}}
\nc{\ICAdd}{\Add_{\mathsf{ic}}}%
\nc{\ICADD}{\ADD_{\mathsf{ic}}}%
\nc{\Idcat}[1]{\Id_{\cat{#1}}}
\nc{\id}{\mathrm{id}}
\nc{\Id}{\mathrm{Id}}
\nc{\ie}{{\sl i.e.}\ }
\nc{\into}{\mathop{\rightarrowtail}}
\nc{\inv}{^{-1}}
\nc{\Iout}[1]{\Ivo{\sout{#1}}}
\nc{\isocell}[1]{\undersett{ #1}{\overset{\sim}{\Ecell}}} % to be used ONLY for 2-cells in xypic diagrams [I]
\nc{\backisocell}[1]{\undersett{ #1}{\overset{\sim}{\Wcell}}} % to be used ONLY for 2-cells in xypic diagrams [I]
\nc{\Isocell}[1]{\undersett{ #1}{\overset{\sim}{\Longrightarrow}}}% ONLY inline with target and source [I]
\nc{\isoEcell}{\overset{\sim}{\Rightarrow}} % to be used ONLY in-line, with target and source [I]
\nc{\isotoo}{\stackrel{\sim}\longrightarrow}
\nc{\isoto}{\buildrel \sim\over\to}
\nc{\Ivo}[1]{{\color{OliveGreen}#1}}
\nc{\JJ}{\mathbb{J}}% old \II % the class of `faithful' guys, we can try other things.
\nc{\kk}{\Bbbk}
\nc{\KK}{\mathrm{KK}}
\nc{\leps}{{}^{\ell}\eps}
\nc{\leta}{{}^{\ell}\eta}
\nc{\loccit}{{\sl loc.\ cit.}}
\nc{\lotoo}[1]{\overset{#1}{\,\longleftarrow\,}}
\nc{\loto}[1]{\overset{#1}{\leftarrow}}
\nc{\lto}{\leftarrow}
\nc{\lun}{\mathrm{lun}} % left unitor
\nc{\Mackintro}[1]{(Mack\,\ref{Mack-#1-intro})}
\nc{\Mack}[1]{(Mack\,\ref{Mack-#1})}
\nc{\Mid}{\,\big|\,}
\nc{\MMod}{\,\text{-}\Mod}%
\dmo{\mods}{mod}%
\nc{\mmods}{\,\text{-}\mathrm{mod}}%
\nc{\MM}{\cat{M}}% a Mackey 2-functor
\nc{\Muniv}{\cat{M}_{\mathsf{univ}}}
\nc{\Ncell}{\rotatebox[origin=c]{0}{$\Uparrow$}} % treated as the others, for uniform size
\nc{\NEcell}{\rotatebox[origin=c]{135}{$\Downarrow$}} % North-East oriented 2-cell arrow
\nc{\NN}{\cat{N}}% another Mackey 2-functor
\nc{\noloc}{\nobreak\mspace{6mu plus 1mu}{:}\nonscript\mkern-\thinmuskip\mathpunct{}\mspace{2mu}}% mirror of \colon
\nc{\NWcell}{\rotatebox[origin=c]{-135}{$\Downarrow$}} % North-West oriented 2-cell arrow
\nc{\oEcell}[1]{\overset{\scriptstyle #1}{\Ecell}} % to be used ONLY for 2-cells in xypic diagrams [I]
\nc{\oWcell}[1]{\overset{\scriptstyle #1}{\Wcell}} % same, but for Wcell
\nc{\ointo}[1]{\overset{#1}{\rightarrowtail}}
\nc{\olto}[1]{\overset{#1}\lto}
\nc{\onto}{\mathop{\twoheadrightarrow}}
\nc{\op}{{\mathrm{op}}}
\nc{\xto}[1]{\xrightarrow{#1}}% much simpler!
\nc{\oto}[1]{\overset{#1}\to}
\nc{\Paul}[1]{{\color{Blue}#1}}
\nc{\pih}[1]{\tau_{1}#1}%
\nc{\Pout}[1]{\Paul{\sout{#1}}}
\nc{\PsFunJindex}{\PsFun_{{\JJ_!}} \ \ {{\JJ}_{!}}\textrm{-strong pseudo-functors}}% hideous trick to circumvent the index problem with \PsFunJ
\nc{\qquadtext}[1]{\qquad\textrm{#1}\qquad}
\nc{\quadtext}[1]{\quad\textrm{#1}\quad}
\nc{\ra}{\rightarrow}
\nc{\reps}{{}^{r\!}\eps}
\nc{\restr}[1]{{|_{\scriptstyle #1}}}% to restrict a map
\nc{\reta}{{}^{r\!}\eta}
\nc{\run}{\mathrm{run}} % right unitor
\nc{\Sad}{\mathsf{Sad}}
\nc{\SAD}{\mathsf{SAD}}
\nc{\sbull}{{\scriptscriptstyle\bullet}}
\nc{\Scell}{\rotatebox[origin=c]{0}{$\Downarrow$}} % treated as the others, for uniform size
\nc{\SEcell}{\rotatebox[origin=c]{45}{$\Downarrow$}} % South-East oriented 2-cell arrow
\nc{\SET}[2]{\big\{\,#1\Mid#2\,\big\}}
\nc{\set}{\mathsf{set}} % category of finite sets
\nc{\Set}{\mathsf{Set}}% category of sets
\nc{\smallfaithful}{\mathsf{f}}% faithful
\nc{\smallslash}{{}^{\scriptscriptstyle/}}
\nc{\smat}[1]{\left(\begin{smallmatrix} #1 \end{smallmatrix}\right)}
\nc{\spanG}{{\widehat{\mathsf{gp}\,\,}\!\!\mathsf{d}}{}^\smallfaithful_{\!{}^{\scriptscriptstyle/}\!G}}% span of groupoids over G.
\nc{\Spanhat}{\textrm{\sf S}\widehat{\textrm{\sf pan}}} %
\nc{\Span}{\Spanname}% the span bicategory of a (2,1)-category
\nc{\sset}{\textrm{-}\set}
\nc{\str}{\mathsf{str}}
\nc{\SWcell}{\rotatebox[origin=c]{-45}{$\Downarrow$}} % South-West oriented 2-cell arrow
\nc{\too}{\mathop{\longrightarrow}\limits}
\nc{\tristars}{\begin{center} $ *\;*\;* $ \end{center}}
\nc{\tSpan}{\pih{\Spanname}}% 1-truncation of Span
\nc{\undersett}[1]{\underset{\scriptstyle #1}}
\nc{\un}{\mathrm{un}} % "unitor" of pseudo-functors
\nc{\vcorrect}[1]{{\vphantom{\vbox to #1em{}}}}
\nc{\Wcell}{\rotatebox[origin=c]{90}{$\Uparrow$}} % treated as the others, for uniform size
\nc{\what}[1]{\widehat{\cat{#1}}}% span category with name #1.
\nc{\xra}{\xrightarrow}

\nc{\xBur}{\mathrm{B^c}} % crossed Burnside ring
\nc{\xBurk}{ \mathrm{B}^{\mathrm{c}}_{\kk} } % crossed Burnside k-algebra
\nc{\Bur}{\mathrm{B}} % ordinary Burnside ring
\nc{\Burk}{\Bur_{\kk}} % ordinary Burnside k-algebra

\nc{\isoTo}{\overset{\sim}{\Rightarrow}}
\nc{\isoc}[3]{#1\,{\diamond}_{_{\!#3}}#2}
\nc{\Isoc}[3]{(\isoc{#1}{#2}{#3})}

\begin{document}

%------------------------------------------------------------------------------

\title{Green equivalences in equivariant mathematics}

\author{Paul Balmer}
\author{Ivo Dell'Ambrogio}
\date{\today}

\address{\ \vfill
\noindent PB: UCLA Mathematics Department, Los Angeles, CA 90095-1555, USA}
\email{balmer@math.ucla.edu}
\urladdr{http://www.math.ucla.edu/$\sim$balmer}

\address{\ \medbreak
\noindent ID: Univ.\ Lille, CNRS, UMR 8524 - Laboratoire Paul Painlev\'e, F-59000 Lille, France}
\email{ivo.dell-ambrogio@univ-lille.fr}
\urladdr{http://math.univ-lille1.fr/$\sim$dellambr}

\begin{abstract} \normalsize
We establish Green equivalences for all Mackey 2-functors, without assuming Krull-Schmidt.
By running through the examples of Mackey 2-functors, we recover all variants of the Green equivalence and Green correspondence known in representation theory and obtain new ones in several other contexts. Such applications include equivariant stable homotopy theory in topology and equivariant sheaves in geometry.
\end{abstract}

\thanks{First-named author supported by NSF grant~DMS-1901696. Second-named author supported by Project ANR ChroK (ANR-16-CE40-0003) and Labex CEMPI (ANR-11-LABX-0007-01).}


\maketitle

%------------------------------------------------------------------------------

\tableofcontents

%------------------------------------------------------------------------------
\vskip-\baselineskip\vskip-\baselineskip\vskip-\baselineskip
%------------------------------------------------------------------------------
%
\section{Introduction}
\label{sec:introduction}%
%\bigbreak
%------------------------------------------------------------------------------

The \emph{Green correspondence} \cite{Green59,Green64} is one of the fundamental and most useful results in modular representation theory of finite groups.
In its simplest form, it says that if~$\kk$ is a field of characteristic~$p$ and if $D\leq H\leq G$ are finite groups such that $D$ is a $p$-group and $H$ contains the normalizer~$N_G(D)$ of~$D$, then the induction and restriction functors yield a bijection between the isomorphism classes of indecomposable $\kk$-linear representations of~$G$ with vertex~$D$ and the isomorphism classes of indecomposable representations of~$H$ with same vertex~$D$. The \emph{vertex} of a representation $M$ is the smallest subgroup $D$ such that $M$ is a retract of a representation induced from~$D$. The Green correspondence reduces many questions to `$p$-local' representation theory, \cf \cite{Alperin80}.

Green~\cite{Green72} later showed that his correspondence follows easily by tracking indecomposable objects through what is now called the \emph{Green equivalence}
\begin{equation} \label{eq:intro-Greeneq}
\Ind_H^G \colon \frac{\mods(\kk H ; D)}{\mods(\kk H; \mathfrak{X})}
\isotoo
\frac{\mods(\kk G ; D)}{\mods(\kk G; \mathfrak{X})}
\end{equation}
an equivalence of categories between additive subquotients of the module categories; see \Cref{Rec:add-quotient}. Here $\mods(\kk G;D)\subseteq \mods(\kk G)$ is the full subcategory of all retracts of finitely-generated $\kk G$-modules induced up from~$D$, and similarly for $\mods(\kk G;\mathfrak{X})$ by allowing induction from all subgroups in the family
\begin{equation}
\label{eq:X}%
\mathfrak{X}=\mathfrak{X}(G,H,D):=\{ D \cap {}^{g\!} D \mid g\in G\smallsetminus H\}
\end{equation}
(where of course ${}^{g\!}D=gDg\inv$). Although Green's proof still used the Krull-Schmidt property of the categories of finite-dimensional modules, the statement of the Green equivalence~\eqref{eq:intro-Greeneq} makes sense more generally, \eg for infinite-dimensional representations. Indeed, the result was eventually extended to that case in~\cite{BensonWheeler01}. The recent preprint \cite{CarlsonWangZhang20pp} further extends the Green equivalence and correspondence to various derived categories of chain complexes of representations.

\tristars

In this paper we show that, in fact, the Green equivalence is not specific to $\kk$-linear representation theory. It is a general fact about `equivariant mathematics', a necessary consequence of nothing more than having induction and restriction satisfying some basic adjunction and Mackey-style relations, as commonly found throughout mathematics. Let us explain this idea.

Fix a finite group~$G$. To obtain a Green equivalence, we only need a \emph{Mackey 2-functor for~$G$} in the sense of~\cite{BalmerDellAmbrogio18pp}. This algebraic gadget consists of an additive category $\cat M(H)$ for each subgroup $H\leq G$, together with `induction' and `restriction' functors $\Ind_K^H \colon \cat M(K) \leftrightarrows \cat M(H) \,:\Res^H_K $ for all $K\leq H\leq G$, as well as conjugation functors and conjugation natural isomorphisms; this structure satisfies natural relations, most notably induction and restriction are adjoint on both sides and satisfy a suitably categorified version of the Mackey formula; see \Cref{sec:recollections} for details. None of those very general relations are mysterious and they have been in common use long before they were given the name `Mackey 2-functor' in~\cite{BalmerDellAmbrogio18pp}.

For any such Mackey 2-functor $\MM$ and for any subgroups $K\le H$ we may define
\[
\MM(H;K):= \add (\Ind^H_K(\MM(K))) \subseteq \MM(H)\,,
\]
the \emph{category of $K$-objects} in~$\MM(H)$, namely the full subcategory  of $\MM(H)$ containing all retracts of images of the induction functor~$\Ind_K^H$ (\cf \Cref{Def:D-object}). Similarly for~$\mathfrak{X}$-objects: $\MM(H;\mathfrak{X})=\add\big(\cup_{K\in \mathfrak{X}}\MM(H;K)\big)$.
Our general Green equivalence looks as follows:
\begin{Thm} [\Cref{Thm:Green-equivalence}]
\label{Thm:intro-Green-eq}%
For any Mackey 2-functor $\cat M$ for~$G$ (\Cref{Rec:2-Mackey}) and any subgroups $D\le H\le G$, induction $\Ind_H^G$ yields an equivalence
\begin{equation*}
\bigg(\frac{\MM(H ; D)}{\MM(H; \mathfrak{X})}\bigg)^\natural
\isotoo
\bigg(\frac{\MM(G ; D)}{\MM(G; \mathfrak{X})}\bigg)^\natural
\end{equation*}
between the idempotent-completions of the additive quotient categories of $D$-objects in $\MM(H)$ and~$\MM(G)$, by those of $\mathfrak{X}$-objects where $\mathfrak{X}=\{ D \cap {}^{g\!} D \mid g\in G\smallsetminus H\}$.
\end{Thm}

We remind the reader of idempotent-completion in \Cref{Def:idempotent-completion}. Another formulation of the above is to say that induction $\Ind_H^G$ yields a \emph{fully-faithful} functor
\begin{equation}
\label{eq:intro-Green-eq-up-to-retracts}%
\frac{\MM(H ; D)}{\MM(H; \mathfrak{X})}
\hook
\frac{\MM(G ; D)}{\MM(G; \mathfrak{X})}
\end{equation}
and every object in the right-hand category is a retract of an object on the left. In other words, it is an \emph{equivalence-up-to-retracts} (\Cref{Def:up-to-retracts}).

As said, our theory does not assume Krull-Schmidt. Note that the price we pay to extend Green beyond Krull-Schmidt is very mild: We only need idempotent-completion. Furthermore, for triangulated categories with coproducts, we can get rid of the idempotent-completion altogether and~\eqref{eq:intro-Green-eq-up-to-retracts} is already an equivalence on the nose (see \Cref{Prop:idempotent-complete}). This applies to the first examples we list below.

There is another case where all quotients in sight are idempotent-complete and~\eqref{eq:intro-Green-eq-up-to-retracts} is an equivalence. It is when the categories $\cat M(H)$ do satisfy Krull-Schmidt. In that setting, we can speak of vertices as in representation theory (see \Cref{Rem:vertices-and-sources}) and we obtain a generalized Green correspondence:

\begin{Thm} [{\Cref{Cor:Green-corr}, \Cref{Prop:vertices-match} and \Cref{Cor:Green-corr-vertex}}]
\label{Thm:intro-Green-corr}%
If $\MM$ is a\break Mackey 2-functor for~$G$ taking values in Krull-Schmidt categories, the functors $\Ind_H^G$ and $\Res_H^G$ yield a bijection ``\`a la Green'' between isomorphism classes of those indecomposable objects in $\cat M(H;D)$, respectively in~$\cat M(G;D)$, whose vertex is not $G$-subconjugate to a subgroup in~$\mathfrak{X}$.
The bijection is vertex-preserving, and if $N_G(D)\leq H$, it restricts on both sides to indecomposable objects with vertex~$D$.
\end{Thm}

\tristars

By specializing the above results to the multitude of readily available Mackey 2-functors~$\MM$, we obtain Green equivalences to every equivariant heart's delight.
For instance, in topology, we may consider equivariant stable homotopy theory:

\begin{Cor} [Apply \Cref{Thm:intro-Green-eq} to \Cref{Exa:top}]
\label{Cor:intro-top}%
For each subgroup $H\leq G$, consider $\MM(H):=\SH(H)= \Ho(\mathcal Sp^H)$ the stable homotopy category of genuine $H$-spectra. Then for $D\le H\le G$, the functor $\Ind_H^G = G_+ \wedge_H -$ yields an equivalence
\begin{equation*}
\frac{\SH(H ; D)}{\SH(H; \mathfrak{X})}
\isotoo
\frac{\SH(G ; D)}{\SH(G; \mathfrak{X})},
\end{equation*}
between additive quotient categories of $D$-objects, as above.
\end{Cor}

In noncommutative geometry, we may consider equivariant KK-theory:

\begin{Cor} [Apply \Cref{Thm:intro-Green-eq} to \Cref{Exa:KK}]
\label{Cor:intro-KK}
For each $H\leq G$, consider $\MM(H):=\KK(H):= \KK^H$ the Kasparov category of separable complex $H$-C*-algebras. Then for $D\le H\le G$, the functor $\Ind_H^G= {\rm C}(G,-)^H$ yields an equivalence
\begin{equation*}
\frac{\KK(H ; D)}{\KK(H; \mathfrak{X})}
\isotoo
\frac{\KK(G ; D)}{\KK(G; \mathfrak{X})}
\end{equation*}
between additive quotient categories of $D$-objects, as above.
\end{Cor}

In algebraic geometry, we may consider equivariant sheaves over varieties. In this setting we can even find Mackey 2-functors satisfying Krull-Schmidt:
\begin{Cor}  [Apply Theorems \ref{Thm:intro-Green-eq} and \ref{Thm:intro-Green-corr} to \Cref{Exa:sheaves}]
\label{Cor:intro-geom}%
Let $X$ be a regular and proper (\eg smooth projective) algebraic variety over a field~$\kk$, equipped with an algebraic action of~$G$. For each subgroup $H\leq G$ consider $\MM(H):=\Db(X/\!\!/H)$, the bounded derived category of coherent $H$-equivariant sheaves on~$X$.
Then for $D\le H\le G$, the induction functor $\Ind_H^G$ yields an equivalence
\begin{equation*}
\frac{\Db(X/\!\!/ H ; D)}{\Db(X/\!\!/H; \mathfrak{X})}
\isotoo
\frac{\Db(X /\!\!/ G ; D)}{\Db(X /\!\!/ G; \mathfrak{X})}
\end{equation*}
of $\kk$-linear categories. If $N_G(D)\leq H$, the above yields a bijection
\[
\left\{{\textrm{indecomposable objects in} \atop \Db(X/\!\!/H) \textrm{ with vertex }D}\right\}
\overset{\sim}{\longleftrightarrow}
\left\{{\textrm{indecomposable objects in } \atop \Db(X/\!\!/G) \textrm{ with vertex }D}\right\}
\]
of isomorphism classes of indecomposable complexes of equivariant sheaves.
\end{Cor}

As in representation theory, which is the special case $X=\Spec(\kk)$ of the above, the bijection in \Cref{Cor:intro-geom} is non-trivial only when the field~$\kk$ has positive characteristic~$p$ and $D$ is a $p$-group.

We trust the reader gets the idea from the above sample: Such applications are limitless. We also easily recover all versions of the correspondence known in representation theory; see \Cref{Exa:Green-corr-modrep} and the following remarks. In particular, we obtain the Green equivalence for derived categories of (unbounded) chain complexes, and the Green correspondence for indecomposable complexes.

\begin{Rem}
The theory of Mackey 2-functors, as developed in \cite{BalmerDellAmbrogio18pp} and used in this article, is formulated in terms of finite \emph{groupoids}, rather then just finite groups. This is not a gratuitous generalization but is done out of convenience, for instance because groupoids allow us to formulate \emph{canonical} Mackey formulas, without having to choose any coset representatives. This uses the notion of \emph{Mackey square} and is briefly recalled in \Cref{sec:recollections}. In the present article, we also use Mackey squares to give a unified and conceptual treatment of the somewhat mysterious classes of subgroups traditionally denoted $\mathfrak{X}, \mathfrak{Y}$ and the like, that typically come up in the proof of the Green correspondence; see \Cref{sec:partial}.
\end{Rem}

\begin{Rem}
At least two works, \cite{AuslanderKleiner94} and \cite{CarlsonWangZhang20pp}, prove their representation-theoretic versions of the Green correspondence by way of some abstract results on adjoint functors, which do not even mention finite groups. Of course, these are \emph{not} recovered by our equivariant methods. However \cite{AuslanderKleiner94} is rather complicated and hard to relate to the examples. Although a significant improvement over~\cite{AuslanderKleiner94}, the recent \cite[\S\,6]{CarlsonWangZhang20pp} still involves a big diagram of categories and a list of several technical conditions, which are not trivial to understand intuitively. On the other hand, it is very simple to derive new Green equivalences with our approach, because it is easy to produce new examples of Mackey 2-functors. Moreover, every reader can remember the concept of ``additive categories $\MM(G)$ depending on finite groups~$G$, with induction, restriction and a Mackey formula".
\end{Rem}

\begin{Rem}
As alluded to above, the quotient categories appearing in the Green correspondence~\eqref{eq:intro-Greeneq} sometimes have more structure than just the additive one. Notably, they are often subcategories of triangulated categories, with somewhat exotic triangulations; see~\cite[Section~7]{Beligiannis00}. See also \Cref{Prop:idempotent-complete}. The reader interested in further details is referred to Zimmermann~\cite{Zimmermann20pp}.
\end{Rem}

The organization of the paper should be clear from the above introduction and the table of contents.

%------------------------------------------------------------------------------
%
\section{Additive preliminaries}
\label{sec:basics}%
%\bigbreak
%------------------------------------------------------------------------------

We recall a few basics and fix some terminology mostly about additive categories.

\begin{Not}
The symbol $\simeq$ denotes isomorphisms. We reserve $\cong$ for natural and canonical isomorphisms.
\end{Not}

\begin{Not}
We will write $x\leq y$ to express the fact that an object $x$ is a retract of an object~$y$, meaning that there are maps $\alpha\colon x\to y$ and $\beta\colon y\to x$ such that $\beta\alpha=\id_x$. In an additive category that is \emph{idempotent-complete} (\ie every idempotent endomorphism $e=e^2\colon y\to y$ yields a splitting $y\simeq \textrm{Im}(e)\oplus \Ker(e)$), an object $x$ is a retract of $y$ if and only if it is a direct summand: $y\simeq x \oplus x'$ for some object~$x'$.
\end{Not}

\begin{Def}
\label{Def:idempotent-completion}%
The idempotent-completion $\cat{C}^\natural$ (\aka Karoubi envelope) is the universal idempotent-complete category receiving~$\cat{C}$. It can be explicitly constructed as pairs $(x,e)$ where $x\in \Obj(\cat{C})$ and $e=e^2\colon x\to x$ is an idempotent, with morphisms $f\colon (x,e)\to (x',e')$ given by $f\colon x\to x'$ in~$\cat{C}$ such that $e'fe=f$. The fully-faithful embedding $\cat{C}\hook \cat{C}^\natural$ maps $x$ to~$(x,\id_x)$.
\end{Def}

\begin{Exa}
The idempotent-completion of the category of free modules is the category of projective modules.
\end{Exa}

\begin{Def}
\label{Def:up-to-retracts}%
A functor $F\colon\cat{C}\to \cat{D}$ is \emph{surjective-up-to-retracts} if every object of~$\cat{D}$ is a retract of $F(x)$ for some object $x$ of~$\cat{C}$. If moreover $F$ is fully faithful, we say that $F$ is an \emph{equivalence-up-to-retracts}. The latter is equivalent to the induced functor $F^\natural\colon \cat{C}^\natural\to \cat{D}^\natural$ on idempotent-completions being an equivalence.
\end{Def}

\begin{Not}
\label{Not:add}%
If $\cat E\subseteq \cat A$ is a collection of objects in an additive category, $\add(\cat E)$ will denote the smallest full subcategory of $\cat A$ containing $\cat E$ and closed under taking finite directs sums and retracts. If $\cat A$ is idempotent-complete then so is $\add(\cat E)$.
\end{Not}

\begin{Rec}
\label{Rec:add-quotient}%
Let $\cat{B}\subseteq\cat{A}$ be a full additive subcategory of an additive category~$\cat{A}$. The \emph{additive quotient} $\cat{A}\onto \cat{A}/\cat{B}$ is the universal additive functor mapping~$\cat{B}$ to zero. It is realized by keeping the same objects as~$\cat{A}$ and taking the following quotients of abelian groups as hom groups:
\[
\cat{A}/\cat{B}(x,y)=\frac{\cat{A}(x,y)}{\left\{\alpha\colon x\to y\ \Big|\ \exists \ \vcenter{\xymatrix@C=1em@R=.6em{x\ar[rr]^-{\alpha}\ar@{..>}[rd]&& y\\ & z \ar@{..>}[ru]}}\textrm{ with }z\in\cat{B}\right\}}
\]
Composition is well-defined on representatives: $[\beta]\circ[\alpha]=[\beta\alpha]$. Note that this construction does not change if~$\cat{B}$ is replaced by~$\add(\cat{B})$, so we can as well assume~$\cat{B}$ closed under retracts in~$\cat{A}$.
\end{Rec}

\begin{Rem}
\label{Rem:retract-in-quotient}%
An object $x\in\cat{A}$ is a retract of~$y\in\cat{A}$ in the quotient $\cat{A}/\cat{B}$ if and only if $x$ is a retract of $y\oplus z$ in~$\cat{A}$ for some~$z\in\cat{B}$. Indeed, let $\alpha\in\cat{A}(x,y)$ and $\beta\in\cat{A}(y,x)$ be such that $[\beta\alpha]=\id_{x}$ in~$\cat{A}/\cat{B}$. This means there exists $z\in\cat{B}$ and a factorization $\id_x-\beta\alpha=\delta\gamma$ for $\gamma\in\cat{A}(x,z)$ and $\delta\in\cat{A}(z,x)$. But then the composite
\[
\xymatrix{x \ar[r]^-{\smat{\alpha\\\gamma}}&y\oplus z \ar[r]^-{\smat{\beta & \delta}}&x}
\]
is the identity of~$x$. The converse is obvious since $z\cong 0$ in~$\cat{A}/\cat{B}$ for all~$z\in\cat{B}$.
\end{Rem}

Let us say a word about situations where those quotients are idempotent-complete. Recall first the following very convenient fact:
\begin{Prop}[{B\"ockstedt-Neeman~\cite[Proposition~3.2]{BoekstedtNeeman93}}]
\label{Prop:BN}%
Let $\cat{T}$ be a triangulated category admitting countable coproducts. Then $\cat{T}$ is idempotent-complete.
\end{Prop}

It therefore becomes interesting to know when quotients admit a triangulation.
\begin{Prop}[{Beligiannis~\cite[Section~7]{Beligiannis00} or~\cite{BalmerStevenson20pp}}]
\label{Prop:triangulation-T/S}%
Let $R\colon \cat{T}\to \cat{S}$ be an exact functor of triangulated categories admitting a two-sided adjoint $I\colon \cat{S}\to \cat{T}$. Then the additive quotient $\cat{T}/\add(I(\cat{S}))$ admits a triangulated structure.
\end{Prop}

\begin{Rem}
Two comments are in order. First, the above quotient is \emph{not} a Verdier quotient. Second, the canonical functor $\cat{T}\onto \cat{T}/\add(I(\cat{S}))$ is usually not exact; it does not even commute with suspension in general. So the triangulated structure on $\cat{T}/\add(I(\cat{S}))$ is somewhat exotic. But the beauty of \Cref{Prop:BN} is that there is no assumption made on the triangulated structure: any one will do.
\end{Rem}

%------------------------------------------------------------------------------
%
\section{Mackey squares and Mackey 2-functors}
\label{sec:recollections}%
%\bigbreak
%------------------------------------------------------------------------------

%
\begin{Rem}
As indicated in the Introduction, instead of finite groups we use finite groupoids, \ie finite categories in which every morphism is an isomorphism. Every finite group~$G$ is seen as a finite groupoid~$G$ with one object and, up to equivalence, a finite groupoid is simply a disconnected union of groups. We denote by $\gpd$ the 2-category of finite groupoids, functors (1-morphisms) and natural transformations (2-morphisms). We often speak of \emph{morphisms} of groupoids $u\colon H\to G$ instead of functors, because there are many other functors around and also to evoke the special case of group homomorphisms. We write $H\into G$ to indicate faithfulness.
\end{Rem}

The us recall the (iso)comma construction in~$\gpd$, which is a 2-categorical variation on the concept of pullback.
\begin{Rec}
\label{Rec:isocomma}%
Let $i\colon H\to G$ and $u\colon K\to G$ be two morphisms of finite groupoids with same target. The \emph{isocomma} groupoid~$\isoc HKG$, also denoted~$(i/u)$
\begin{equation}
\label{eq:isocomma}%
\vcenter{
\xymatrix@C=14pt@R=14pt{
& \isoc HKG \ar[dl]_-{\pr_1} \ar[dr]^-{\pr_2}
\\
H \ar[dr]_-{i} \ar@{}[rr]|-{\isocell{\gamma_{H,K}}}
&& K \ar[dl]^-{u}
\\
& G
}}
\end{equation}
is defined by letting the objects of~$\isoc HKG$ consist of triples $(x,y,g)$ where $x$ and $y$ are objects of~$H$ and~$K$ respectively and $g\colon i(x)\isoto u(y)$ is an isomorphism in~$G$; morphisms $(x,y,g)\to (x',y',g')$ consist of pairs $(h,k)$ where $h\colon x\to x'$ and $k\colon y\to y'$ are morphisms in~$H$ and~$K$ such that $g'i(h)=u(k)g$. This groupoid $\isoc HKG$ comes equipped with two morphisms~$\pr_1$ and~$\pr_2$ as in~\eqref{eq:isocomma}, namely the obvious projections onto the~$H$- and~$K$-parts. Finally the 2-cell $\gamma_{H,K}\colon i\circ\pr_1\isoTo u\circ\pr_2$ is given by~$(\gamma_{H,K})_{(x,y,g)}=g$. This construction enjoys a universal property, see~\cite[\S\,2.1]{BalmerDellAmbrogio18pp}. In particular, any 2-cell
\begin{equation}
\label{eq:Mackey-square}%
\vcenter{
\xymatrix@C=14pt@R=14pt{
& L \ar[dl]_-{v} \ar[dr]^-{j}
\\
H \ar[dr]_-{i} \ar@{}[rr]|-{\isocell{\gamma}}
&& K \ar[dl]^-{u}
\\
& G
}}
\end{equation}
induces a functor $\langle v,j,\gamma\rangle\colon L\to \isoc HKG$, mapping $z\in L$ to~$(v(z),j(z),\gamma_z)$. When this functor $L\to \isoc HKG$ is an equivalence the square~\eqref{eq:Mackey-square} is called a \emph{Mackey square}; such squares enjoy a (weaker) universal property, see \cite{BalmerDellAmbrogio18pp} for details.
\end{Rec}

\begin{Rem}
Instead of the 2-categorical notation~$(i/u)$ adopted in~\cite{BalmerDellAmbrogio18pp}, we systematically write $\isoc HKG$ in this paper to avoid conflict with additive quotients. Writing $\isoc HKG$ requires the morphisms~$i$ and~$u$ to be unambiguous from context.
\end{Rem}

\begin{Exa}
\label{Exa:isoc-triv}%
We have canonical equivalences $H\cong\isoc HGG$ and $H\cong \isoc GHG$ given respectively by $\langle \Id_H,i,\id_i\rangle$ and~$\langle i,\Id_H,\id_i\rangle$.
\end{Exa}

\begin{Exa}[{\cite[Remark~2.2.7]{BalmerDellAmbrogio18pp}}]
\label{Exa:groupist-translation}%
Here is the link with the `usual' Mackey double-coset formula. Let $H,K\le G$ be subgroups of a finite group, considered as one-object groupoids. Even in that case, $\isoc HKG$ is usually not connected, \ie it does not boil down to a single group. In fact, $\isoc HKG$ has one connected component for each class in~$\doublequot HGK$. The choice of a representative in each such coset (\ie one object in each component) yields a non-canonical equivalence
\[
\coprod_{[g]\in \doublequot HGK} \! H\cap{}^g K
\ \isotoo \ \isoc HKG
\]
that can be used to turn our canonical Mackey formulas for groupoids into the more traditional but non-canonical Mackey double-coset formulas for groups.
\end{Exa}

\begin{Lem} \label{Lem:MS-sum}
The isocomma construction commutes with coproducts, in the sense that for any $H,K,L\rightarrowtail G$ there are canonical isomorphisms of groupoids
\[
\Isoc{H}{(K\sqcup L)}{G} \cong \Isoc H K G \sqcup \Isoc H L G
\quad \textrm{ and } \quad
\Isoc{(K\sqcup L)}{H}{G} \cong \Isoc  K H G \sqcup \Isoc L H G
\]
which are compatible with the projections and the 2-cells.
\end{Lem}

\begin{proof}
Easy exercise.
\end{proof}

\begin{Lem} \label{Lem:MS-2-out-of-3}
In a configuration of square 2-cells as follows
\[
\xymatrix@C=18pt@R=14pt{
&& \ar[dl] \ar[dr] & \\
& \ar[dl] \ar[dr] \ar@{}[rr]|{\isocell{\sigma}} && \ar[dl] \\
\ar[dr] \ar@{}[rr]|{\isocell{\gamma}} && \ar[dl]& \\
&&&
}
\]
if both $\gamma$ and the composite square are Mackey squares, then so is~$\sigma$.
\end{Lem}

\begin{proof}
Straightforward from the universal property of a Mackey square; see~\cite[Definition~2.1.1 and \S\,2.2]{BalmerDellAmbrogio18pp}.
\end{proof}

\begin{Not} \label{Not:strict-fun-isocommas}
Given two isocomma squares
\[
\vcenter{
\xymatrix@C=14pt@R=14pt{
& \isoc EFG \ar[dl]_-{\pr_1} \ar[dr]^-{\pr_2}
\\
E \ar[dr]_-{u} \ar@{}[rr]|-{\isocell{\gamma}}
&& F \ar[dl]^-{v}
\\
& G
}}
\qquadtext{and}
\vcenter{
\xymatrix@C=14pt@R=14pt{
& \isoc {E'}{F'}{G'} \ar[dl]_-{\pr_1} \ar[dr]^-{\pr_2}
\\
E' \ar[dr]_-{u'} \ar@{}[rr]|-{\isocell{\gamma'}}
&& F' \ar[dl]^-{v'}
\\
&G'
}}
\]
as well as morphisms $i\colon G\to G'$, $k\colon E\to E'$ and $\ell\colon F\to F'$ making the diagram
\[
\xymatrix@C=18pt@R=4pt{
E \ar[dr]^u \ar[dd]_-k & & F \ar[dl]_v \ar[dd]^-\ell \\
& G \ar[dd]_-i & \\
E' \ar[dr]_{u'} & & F' \ar[dl]^{v'} \\
&G' &
}
\]
(strictly) commute, we will write
\[
\Isoc{k}{\ell}{i} \colon \Isoc{E}{F}{G} \too \Isoc{E'}{F'}{G'}
\]
for the canonical morphism with components $\langle  k \pr_1 , \ell \pr_2 , i \gamma  \rangle $. Explicitly, this functor $\isoc{k}{\ell}{i}$ maps $(x,y,g)$ to $(k(x),\ell(y),i(g))$. If the functor $i$ happens to be the identity $i= \Id_G$, we will write $\isoc{k}{\ell}{G}$, and similarly with $k$ and~$\ell$.
\end{Not}

\tristars

\begin{Not}
\label{Not:groupoids}%
It is convenient to consider other 2-categories of groupoids, like $\groupoidf$ the subcategory of~$\gpd$ in which we only take faithful morphisms~$H\into G$. More interesting is the 2-category $\gpdGzero$ of~\cite[Definition~B.0.6]{BalmerDellAmbrogio18pp} consisting of groupoids $(G,i_G)$ together with a chosen faithful embedding $i_G\colon G\into G_0$ into a fixed `ambient' groupoid~$G_0$. This allows us to treat in the same breath the `global' theory for the 2-category~$\GG=\gpd$ (or $\GG=\groupoidf$) and the `$G_0$-local' theory for a given~$G_0$ by using~$\GG=\gpdGzero$. In glorious generality, $\GG$ could be any 2-category as in~\cite[Hypothesis~5.1.1]{BalmerDellAmbrogio18pp} but the reader can keep one of the above in mind:
\[
\GG=\groupoidf\qquadtext{or}
\GG=\gpdGzero.
\]
\end{Not}

\begin{Rec}
\label{Rec:2-Mackey}%
Let $\GG$ be our 2-category of finite groupoids of interest, as above. A \emph{Mackey 2-functor} $\MM\colon \GG^{\op}\to \ADD$ is a strict 2-functor taking values in additive categories and additive functors satisfying the following four axioms:
\begin{enumerate}[(\textrm{Mack}~1)]
\item
Additivity: $\MM(G_1\sqcup G_2)=\MM(G_1)\oplus \MM(G_2)$.
\smallbreak
\item
Existence of adjoints: For every faithful $i\colon H\into G$, the restriction functor $i^*=\MM(i)\colon \MM(G)\to \MM(H)$ admits a left adjoint $i_!$ and a right adjoint~$i_*$.
\smallbreak
\item
For every Mackey square~\eqref{eq:Mackey-square}, the mates of the 2-cell~$\gamma^*$ with respect to the adjunctions of (Mack~2) give Base-Change isomorphisms
\[
j_!\circ v^* \isoTo u^* \circ i_!
\qquadtext{and}
u^*\circ i_* \isoTo j_*\circ v^*\,.
\]
\smallbreak
\item
Ambidexterity: Induction and coinduction coincide: $i_!\simeq i_*$.
\end{enumerate}
If $G_0$ is a fixed finite group, a \emph{Mackey 2-functor for~$G_0$} is one where $\GG=\gpdGzero$.
\end{Rec}

\begin{Rem} \label{Rem:can-Theta}
The reader will find a detailed discussion of these axioms, and their beautification, in~\cite[Remark~1.1.10, \S\,1.2 and Chapter~3]{BalmerDellAmbrogio18pp}.
In particular, if there exists \emph{any} isomorphism $i_!\simeq i_*$ as in (Mack~4), then a certain canonical and well-behaved natural transformation $\Theta_i\colon i_!\Rightarrow i_*$ is also invertible and lets us combine $i_!$ and $i_*$ into a single functor, in the following denoted by~$i_*$, adjoint to $i^*$ on both sides and satisfying further extra properties \cite[Theorem~1.2.1]{BalmerDellAmbrogio18pp}.
\end{Rem}

\begin{Rem}
One virtue of Mackey 2-functors is the profusion of examples extending beyond representation theory; see~\cite[Chapter~4]{BalmerDellAmbrogio18pp} or~\Cref{sec:examples}.
\end{Rem}

%------------------------------------------------------------------------------
%
\section{The operator~$\partial$ on groupoids}
\label{sec:partial}%
%\bigbreak
%------------------------------------------------------------------------------

Let us fix a faithful morphism between two finite groupoids
\[H\ointo{i} G.\]
The construction~$\partial_i(-,-)$ introduced in this section measures the difference between isocommas $\isoc{-}{-}{H}$ and~$\isoc{-}{-}{G}$. See \Cref{Rec:isocomma} and \Cref{Not:strict-fun-isocommas}.

\begin{Prop}
\label{Prop:partial-prep}%
Let $E\ointo{\iota_E}H$ and $F\ointo{\iota_F}H$ be faithful and consider the isocommas
\begin{equation}
\label{eq:isoc-FDH-FDG}%
\vcenter{
\xymatrix@C=14pt@R=14pt{
& \isoc EFH \ar[dl]_-{\pr_1} \ar[dr]^-{\pr_2}
\\
E \ar[dr]_-{\iota_E} \ar@{}[rr]|-{\isocell{\kappa_{E,F}}}
&& F \ar[dl]^-{\iota_F}
\\
& H
}}
\qquadtext{and}
\vcenter{
\xymatrix@C=14pt@R=14pt{
& \isoc EFG \ar[dl]_-{\pr_1} \ar[dr]^-{\pr_2}
\\
E \ar[dr]_-{i\iota_E} \ar@{}[rr]|-{\isocell{\gamma_{E,F}}}
&& F.\! \ar[dl]^-{i\iota_F}
\\
&G
}}
\end{equation}
Then the functor $\Isoc EF{i}\colon \Isoc EFH\to \Isoc EFG$ is fully-faithful.
\end{Prop}

\begin{proof}
The functor $\isoc EF{i}$ maps an object $(x,y,h)$ to $(x,y,i(h))$, for all~$x\in \Obj (E)$, $y\in \Obj(F)$ and $h\colon \iota_E(x)\isoto \iota_F(y)$ in~$H$, and it is the identity on morphisms. More precisely, morphisms $(x,y,h)\to (x',y',h')$ in~$\isoc EFH$ and morphisms $(x,y,i(h))\to (x',y',i(h'))$ in~$\isoc EFG$ both consist of pairs of morphisms $(e\in E(x,x'),f\in F(y,y'))$ such that, respectively, $h'\iota_E(e)=\iota_F(f)h$ or $i(h')i\iota_E(e)=i\iota_F(f)i(h)$. But these two conditions are equivalent since $i\colon H\into G$ is faithful.
\end{proof}

\begin{Rem}
In fact, the image of~$\isoc EFH$ in~$\isoc EFG$ is even \emph{replete}, \ie closed under isomorphisms. So it is exactly a union of connected components of~$\isoc EFG$.
\end{Rem}

\begin{Not}
\label{Not:partial}%
Let $E,F\into H$ as above. By \Cref{Prop:partial-prep}, the groupoid $\isoc EFH$ is equivalent to a union of connected components of~$\isoc EFG$. We denote by
\[
\partial_i(E,F)
\]
the full subgroupoid of~$\isoc EFG$ consisting of the union of the connected components which do not meet the image of~$\isoc EFH$. By construction, we thus have
\begin{equation}
\label{eq:partial-def}%
\Isoc EFG\ \cong \ \Isoc EFH \ \sqcup \ \partial_i(E,F).
\end{equation}
Of course, the groupoid $\partial_i(E,F)$ not only depends on~$i\colon H\into G$ and on~$E$ and~$F$ but really depends on the embeddings~$\iota_E\colon E\into H$ and $\iota_F\colon F\into H$.
\end{Not}

\begin{Rem}
\label{Rem:partial-proj}%
The two projections~$\pr_1\colon \Isoc EFG\to E$ and $\pr_2\colon \Isoc EFG\to F$ and the 2-cell $\gamma_{E,F}\colon i\iota_E\,\pr_1\Rightarrow i\iota_F\pr_2$ of~\eqref{eq:isoc-FDH-FDG} restrict to~$\partial_i(E,F)$
\[
\xymatrix@C=14pt@R=14pt{
&\partial_i(E,F) \ar[ld]_-{\pr_1} \ar[rd]^-{\pr_2} \ar@{}[dd]|-{\isocell{\gamma_{E,F}}}
\\
E \ar[rd]_-{i\iota_E} && F \ar[ld]^-{i\iota_F}
\\
& G
}
\]
and we keep the same notation for these restrictions when no confusion ensues.
\end{Rem}

\begin{Exas}
\label{Exas:partial-HD}%
In our application, we shall be given three groupoids $D\ointo{j} H\ointo{i} G$. In that setting, there will be three pairs $(E,F)$ to which we need to apply the $\partial_i(E,F)$ construction, namely $(D,D)$, $(H,D)$ and $(H,H)$.
\begin{enumerate}[(1)]
\item
\label{Exa:partial-DD}%
Taking $E=F=D\ointo{j}H$, we have a decomposition
\[
\Isoc DDG\cong \Isoc DDH\sqcup \partial_i(D,D).
\]
The groupoid~$\partial_i(D,D)$ will play an important role in the Green equivalence.
\smallbreak
\item
\label{Exa:partial-HD}%
Taking $E=H$ itself and $F=D\into H$, we have a decomposition
\[
\Isoc HDG\cong \Isoc HDH\sqcup \partial_i(H,D)\cong D\sqcup \partial_i(H,D),
\]
using the equivalence $D\isoto \Isoc HDH$ of \Cref{Exa:isoc-triv}.
\item
\smallbreak
\label{Exa:partial-HH}%
Taking $E=F=H\ointo{\id}H$, we have as a special case of~\eqref{Exa:partial-HD} a decomposition
\[
\Isoc HHG\cong \Isoc HHH\sqcup \partial_i(H,H)\cong H\sqcup \partial_i(H,H).
\]
\end{enumerate}
\end{Exas}

\begin{Rem}
\label{Rem:groupist-translation}%
The reader can keep the following special case in mind throughout the article. Suppose that $D\ointo{j} H\ointo{i} G$ are subgroup inclusions for some good old finite groups $G$, $H$ and~$D$. Then the three $\partial$-constructions of \Cref{Exas:partial-HD} are equivalent to the following coproducts of finite groups (\cf \Cref{Exa:groupist-translation}):
\[
\coprod_{[g]\in \doublequot DGD\atop g\notin H}\!\!\!\!\!\!\!\! D\cap\,{}^g\!D \simeq \partial_i(D,D),
\quad\coprod_{[g]\in \doublequot HGD\atop g\notin H}\!\!\!\!\!\!\!\! H\cap\,{}^g\!D \simeq \partial_i(H,D),
\quad\coprod_{[g]\in \doublequot HGH\atop g\notin H}\!\!\!\!\!\!\!\! H\cap\,{}^g\!H \simeq \partial_i(H,H).
\]
In the representation-theoretic literature, the first two coproducts correspond to two collections of subgroups typically denoted $\mathfrak{X}$ and $\mathfrak{Y}$ respectively; see~\eqref{eq:X}. We shall sometimes write $\mathfrak{U} :=\{H\cap {}^{g\!}H\mid g\in G\smallsetminus H\}$ for the third one.
\end{Rem}

Sooner or later we are going to apply the constructions $\partial_i(E,F)$ to morphisms $E,F\into H$ which are themselves obtained as isocommas or as $\partial_i(E',F')$ for other $E',F'\into H$. Since there are more than one faithful functor out of such more complicated objects into~$H$, we need to specify which one we use.

\begin{Conv}
\label{Conv:partial}%
In all cases, we tacitly embed $\isoc EFG$ into~$H$ via the \emph{first} projection $\Isoc EFG\ointo{\pr_1}E \into H$ and similarly for the subgroupoid~$\partial_i(E,F)$ of~$\isoc EFG$.
\end{Conv}

\begin{Prop}[Diagonality]
\label{Prop:diag}%
Let $E\ointo{k} E'\ointo{\iota_{E'}} H$ with $\iota_E=\iota_{E'}k$ and let\break $F\ointo{\ell} F'\ointo{\iota_{F'}} H$ with $\iota_F=\iota_{F'}\ell$. Consider the groupoids $\partial_i(E,F)$ and $\partial_i(E',F')$ and the associated decompositions as in~\eqref{eq:partial-def}
\[
\Isoc EFG\cong \Isoc EFH \sqcup \partial_i(E,F)
\qquadtext{and}
\Isoc {E'}{F'}G\cong \Isoc {E'}{F'}H \sqcup \partial_i(E',F').
\]
Then the canonical functor~$\isoc k{\ell}G\colon \Isoc EFG\to \Isoc {E'}{F'}G$ induced by~$k\colon E\to E'$ and $\ell\colon F\to F'$ is `diagonal', \ie it respects the above decompositions:
\[
\xymatrix@R=2em{
\isoc EFG \ar@{}[r]|-{\cong} \ar[d]_-{\isoc k{\ell}G}
& \isoc EFH \ar@{}[r]|-{\sqcup} \ar[d]_-{\isoc k{\ell}H}
& \partial_i(E,F) \ar@{..>}[d]_-{\exists\,!}^-{\partial_i(k,\ell)}
\\
\isoc {E'}{F'}G \ar@{}[r]|-{\cong}
& \isoc {E'}{F'}H \ar@{}[r]|-{\sqcup}
& \partial_i(E',F')
}
\]
hence defines a unique functor $\partial_i(k,\ell)\colon \partial_i(E,F)\into \partial_i(E',F')$, which is faithful.
\end{Prop}

\begin{proof}
The functor $\isoc k{\ell}G\colon \Isoc EFG\to \Isoc{E'}{F'}G$ simply maps objects $(x,y,g)$ to $(k(x),\ell(y),g)$. Similarly for $\isoc k{\ell}H\colon \Isoc EFH\to \Isoc{E'}{F'}H$. Hence the following square commutes (on the nose), as can be readily checked:
\[
\xymatrix@R=1.8em{
\isoc EFG \ar[d]_-{\isoc k{\ell}G}
& \ \isoc EFH \ar@{_(->}[l] \ar[d]^-{\isoc k{\ell}H}
\\
\isoc {E'}{F'}G
& \ \isoc {E'}{F'}H \ar@{_(->}[l]
}
\]
Here the horizontal inclusions are the fully-faithful functors of \Cref{Prop:partial-prep}. To show that $\isoc k{\ell}G$ also preserves the `complements', \ie maps $\partial_i(E,F)$ into~$\partial_i(E',F')$, it suffices to show that if the image under~$\isoc k{\ell}G$ of an object $(x,y,g)$ of $\isoc EFG$ is isomorphic to the image of an object of~$\isoc{E'}{F'}H$ inside~$\isoc{E'}{F'}G$ then the given object~$(x,y,g)$ is already equal to the image of an object of~$\isoc EFH$. So suppose that we have $(x',y',h')$ in~$\isoc{E'}{F'}H$ and an isomorphism $(e',f')\colon (k(x),\ell(y),g)\isoto (x',y',i(h'))$ in~$\isoc{E'}{F'}G$ given by~$e'\colon k(x)\isoto x'$ in~$E'$ and $f'\colon \ell(y)\isoto y'$ in~$F$, such that the following square commutes in~$G$:
\[
\xymatrix@R=1.8em{
i\iota_{E'}k(x) \ar[d]_-{i\iota_{E'}(e')} \ar[r]^-{g}
& i\iota_{F'}\ell(y) \ar[d]^-{i\iota_{F'}(f')}
\\
i\iota_{E'}(x') \ar[r]_-{i(h')}
& i\iota_{F'}(y').
}
\]
Let then $h:=\iota_{F'}(f')\inv\circ h'\circ \iota_{E'}(e')\colon \iota_E(x)=\iota_{E'}k(x)\isoto \iota_{F'}\ell(y)=\iota_F(y)$ in~$H$ and note that $i(h)=g$ by the above square. Then the object $(x,y,h)$ of~$\isoc EFH$ maps to our~$(x,y,g)$ in~$\isoc EFG$, as claimed. This proves that~$\isoc k{\ell}G$ is indeed `diagonal'.

Finally, $\partial_i(k,\ell)$ is faithful because so is~$\isoc k{\ell}G$, by faithfulness of~$k$ and~$\ell$.
\end{proof}

\begin{Not}
\label{Not:when-id}%
As usual we simply write $\partial_i(k,F)$ for~$\partial_i(k,\Id_F)\colon$ $\partial_i(E,F)\to \partial_i(E',F)$ and $\partial_i(E,\ell)$ for~$\partial_i(\Id_E,\ell)\colon \partial_i(E,F)\to \partial_i(E,F')$.
\end{Not}

\begin{Exa}
Applying the above in Example~\ref{Exas:partial-HD} we get faithful functors
\[ \partial_i(j,D)\colon \partial_i(D,D)\into \partial_i(H,D)
\quad \textrm{ and } \quad
\partial_i(H,j)\colon \partial_i(H,D)\into \partial_i(H,H) \,.\]
For subgroup inclusions $D\ointo{j} H\ointo{i} G$ as in \Cref{Rem:groupist-translation}, these simply correspond to inclusions $\overline{\mathfrak{X}}\subseteq \overline{\mathfrak{Y}}\subseteq \overline{\mathfrak{U}}$, where for every collection~$\mathfrak{S}$ of subgroups of~$G$ we denote by~$\overline{\mathfrak{S}}$ its closure under taking subgroups: $\overline{\mathfrak{S}}=\SET{K\le G}{\exists\, L\in\mathfrak{S}\textrm{ s.t.\ }K\le L}$.
\end{Exa}

The next statement encapsulates the situation needed in the Green equivalence.

\begin{Lem}
\label{Lem:geography}%
Let $j_1\colon D_1\into H$ and $j_2\colon D_2\into H$. Then we have Mackey squares
\begin{equation}
\label{eq:geography}%
\vcenter{
\xymatrix@C=8pt@R=14pt{
&&&& \Isoc{D_1}{D_2}H \sqcup \partial_i(D_1,D_2) \ar[ld]_(.65){\pr_1\sqcup \partial_i(D_1,j_2)\quad\;} \ar[rd]^(.65){\;\quad\pr_2\sqcup\partial_i(j_1,D_2)} \ar@{}[dd]|(.4){\isoEcell}
\\
&&& \kern-1.5em D_1\sqcup \partial_i(D_1,H) \kern-1.5em \ar[rd]_(.45){ j_1 \sqcup \partial_i(j_1,H) \quad} \ar[llld]_(.65){( \Id \;,\; \pr_1)\;} %\ar@{}[dd]|{=}
&& \kern-1.5em D_2\sqcup \partial_i(H,D_2) \kern-1.5em \ar[ld]^(.45){\quad j_2\sqcup \partial_i(H,j_2)} \ar[rrrd]^(.65){\; ( \Id \;,\; \pr_2)} %\ar@{}[dd]|{=}
\\
D_1 \ar[rrrd]_-{j_1}
&&&& \kern-1.5em H\sqcup \partial_i(H,H) \kern-1.5em \ar[ld]_(.6){( \Id \;,\; \pr_1)\;} \ar[rd]^(.6){\; ( \Id \;,\; \pr_2)} \ar@{}[dd]|{\isoTo}
&&&& D_2 \kern-1.5em \ar[llld]^-{j_2}
\\
&&& H \ar[rd]_-{i}
&& H \ar[ld]^-{i}
\\
&&&& G
}}\kern-1em
\end{equation}
where the top square is a coproduct of two Mackey squares. In particular, we have an equivalence $\Isoc{D_1}{\partial_i(H,D_2)}H\cong \partial_i(D_1,D_2)$.
\end{Lem}

\begin{proof}
Taking the isocommas of $D_1\into H\into G$ against $D_2\into H\into G$  in~$G$, we obtain (using \Cref{Not:strict-fun-isocommas}) the following diagram, where each square is a Mackey square by repeated applications of \Cref{Lem:MS-2-out-of-3}:
\[
\xymatrix@C=14pt@R=14pt{
&& \isoc {D_1}{D_2}G \ar@/_2.5em/[lldd]_-{\pr_1\ } \ar[rd]^(.6){\isoc {j_1}{D_2}G} \ar[ld]_(.6){\isoc {D_1}{j_2}G} \ar@/^2.5em/[rrdd]^-{\pr_2} %\ar@{}[dd]|-{=}
\\
& \isoc{D_1}HG \ar[ld]^-{\pr_1\ } \ar[rd]^-{\isoc {j_1}HG}
&& \isoc H{D_2}G \ar[ld]_-{\isoc H{j_2}G} \ar[rd]_-{\pr_2} %\ar@{}[dd]|{=}
\\
\kern1em D_1 \kern1em \ar[rd]_-{j_1}
&& \isoc HHG \ar[ld]_-{\pr_1} \ar[rd]^-{\pr_2} \ar@{}[dd]|{\isoTo}
&& \kern1em D_2 \kern1em \ar[ld]^-{j_2}
\\
& H \ar[rd]_-{i}
&& H \ar[ld]^-{i}
\\
&& G
}
\]
By unpacking the decompositions~\eqref{eq:partial-def} for all four isocommas and using diagonality as in \Cref{Prop:diag}, we obtain the diagram of Mackey squares in~\eqref{eq:geography}.
Patching together the two upper-left Mackey squares, we see that the top groupoid $\Isoc {D_1}{D_2}H\sqcup \partial_i(D_1,D_2)$ is equivalent to the isocomma $\Isoc{D_1}{(D_2\sqcup \partial_i(H,D_2))}H\cong\Isoc {D_1}{D_2}H\sqcup\Isoc{D_1}{\partial_i(H,D_2)}H$, where the last isomorphism is by \Cref{Lem:MS-sum}.

We still want to verify that this equivalence
\begin{equation} \label{eq:comp-equiv}
\Isoc {D_1}{D_2}H\sqcup \partial_i(D_1,D_2) \cong \Isoc {D_1}{D_2}H\sqcup\Isoc{D_1}{\partial_i(H,D_2)}H
\end{equation}
is diagonal, so as to deduce the claimed equivalence $\partial_i(D_1,D_2)\cong\Isoc{D_1}{\partial_i(H,D_2)}H$ between the right summands.
These being finite groupoids, it suffices to show that the equivalence~\eqref{eq:comp-equiv} restricts to an equivalence on the \emph{left} summands. Indeed, by construction, \eqref{eq:comp-equiv} makes the following diagram commute
\[
\xymatrix@C=0pt{
\Isoc{D_1}{D_2}{H} \sqcup \partial_i(D_1,D_2)
 \ar[rr]_-{\simeq}^-{\textrm{\eqref{eq:comp-equiv}}}
  \ar[dr]^\simeq_{\langle (\pr_1,\pr_1) , \ldots \rangle \quad} &&
\Isoc {D_1}{D_2}H\sqcup\Isoc{D_1}{\partial_i(H,D_2)}H
 \ar[dl]_\simeq^{\quad \langle (\pr_1,\pr_1) , \ldots \rangle} \\
& \Isoc{D_1}{(D_2\sqcup \partial_i(H,D_2))}{H} &
}
\]
where the left equivalence is the canonical comparison of Mackey squares (whose first component $(\pr_1,\pr_1)$ is computed by the top-left composite in \eqref{eq:geography}) and the right equivalence is as in (the proof of) \Cref{Lem:MS-sum}. We deduce from this triangle that \eqref{eq:comp-equiv} restricts to the identity of $\isoc{D_1}{D_2}{H}$.
\end{proof}

We shall need one more, slightly tricky, observation.
\begin{Lem}
\label{Lem:tricky}%
Let $D\ointo{j} H\ointo{i} G$. Then the functor $\pr_1\colon\Isoc H{\partial_i(D,D)}G\into H$ factors via $\pr_1\colon \partial_i(H,D)\into H$.
\end{Lem}

\begin{proof}
We need to construct $u\colon \Isoc H{\partial_i(D,D)}G\to \partial_i(H,D)$ such that $\pr_1\circ u=\pr_1$. As in~\eqref{eq:partial-def}, there is a decomposition of the isocomma $\Isoc H{\partial_i(D,D)}G$ over~$G$ as the disjoint union of the `same' isocomma over~$H$, that is $\Isoc H{\partial_i(D,D)}H\cong\partial_i(D,D)$, with the rest, that is called $\partial_i(H,\partial_i(D,D))$ by definition:
\begin{equation}
\label{eq:aux-decomp-HDDG}%
\Isoc H{\partial_i(D,D)}G \cong \partial_i(D,D)\sqcup\partial_i(H,\partial_i(D,D)).
\end{equation}
The morphism~$u$ is given by two different formulas on those two components.

On the first component, $\partial_i(D,D)$, define $u:=\partial_i(j,D)\colon \partial_i(D,D)\to \partial_i(H,D)$ and note that $\pr_1\circ u=j\pr_1$. The latter is also the restriction of $\pr_1\colon \Isoc H{\partial_i(D,D)}G\to H$ to that component~$\partial_i(D,D)$, as the following diagram is easily seen to commute
\[
\xymatrix@C=6em{
\partial_i(D,D) \ar[r]^-{\langle j\pr_1,\Id,\id_{j\pr_1}\rangle}_-{\simeq} \ar[d]_-{\pr_1}
& \isoc H{\partial_i(D,D)}H \ \ar@{^(->}[r]^-{\isoc H{\partial_i(D,D)}i}_-{\textrm{as in Prop.\,\ref{Prop:partial-prep}}}
& \isoc H{\partial_i(D,D)}G \ar[d]^-{\pr_1}
\\
D \ar[rr]^-{j}
&& H
}
\]
(The top fully-faithful functor is the explicit way $\partial_i(D,D)$ is equivalent to the component $\isoc H{\partial_i(D,D)}H$ of $\isoc H{\partial_i(D,D)}G$.)

On the other component $\partial_i(H,\partial_i(D,D))$ in~\eqref{eq:aux-decomp-HDDG}, we can use the functoriality of~$\partial_i(H,-)$ to define $u:=\partial_i(H,\pr_1)\colon \partial_i(H,\partial_i(D,D))\to \partial_i(H,D)$ for the morphism $\pr_1\colon \partial_i(D,D)\to D$. Note that we are allowed to do this, \ie to apply \Cref{Prop:diag} with $F\into F'\into H$ being $\partial_i(D,D)\ointo{\pr_1} D\ointo{j}H$ as this composite is the (tacit) morphism $\partial_i(D,D)\into H$; see \Cref{Conv:partial}. (This would be inaccurate with $\pr_2$!)
%\Iout{A direct computation gives}
By the construction of $\partial_i(H,\pr_1)$ as a diagonal component of $\isoc{H}{\pr_1}{i}$, we have $\pr_1\circ \partial_i(H,\pr_1)=\pr_1$ as wanted.
\end{proof}

%------------------------------------------------------------------------------
%
\section{The Green equivalence}
\label{sec:Green}%
%\bigbreak
%------------------------------------------------------------------------------

We fix for the section a Mackey 2-functor $\MM\colon \GG^\op\to \ADD$ (\Cref{Rec:2-Mackey}). Recall from \Cref{Not:groupoids} that $\GG$ typically denotes either the 2-category of finite groupoids~$\GG=\groupoidf$ with faithful morphisms or the 2-category $\GG=\gpdGzero$ of groupoids embedded in a given ambient groupoid~$G_0$.

\begin{Def}
\label{Def:D-object}%
Let $j\colon D\into H$ be faithful in~$\GG$. We call an object~$m\in \MM(H)$ a \emph{$D$-object} (or \emph{$j$-object}), if $m$ is a retract of some object of the form $j_*(n)$ with $n\in\MM(D)$. We denote the full subcategory of $D$-objects in~$\cat M(H)$ by
\[
\MM(H;D)=\MM(H;j) := \add (j_*(\cat M(D))) \subseteq \cat M(H)
\]
An object of the form $j_*(n)$ for some $n\in \MM(D)$ will be called a \emph{strict} $D$-object.
\end{Def}

\begin{Rem}
Representation theorists call our $D$-objects \emph{relatively $D$-projective} (for having the left lifting property against $j^*$-split morphisms). We find this terminology cumbersome and not so helpful in our broader context. In view of the ambidextrous adjunction $j_*\adj j^* \adj j_*$, our $D$-objects are also the relatively $D$-\emph{injective} ones (with dual lifting property), or equivalently those~$m$ for which the unit $m\to j_*j^*(m)$ (or the counit $j_*j^*(m)\to m$) has a retraction (resp.\ a section).
\end{Rem}

\begin{Rem}
\label{Rem:partial-trivial}%
Let us get a few elementary observations out of the way:
\begin{enumerate}[(1)]
\item
\label{it:MHD-iso}%
The subcategory $\MM(H;j)$ is independent of the isomorphism class of~$j$, that is, if $j\isoEcell j'\colon D\into H$ then $\MM(H;j)=\MM(H;j')$.
\smallbreak
\item
\label{it:MHE<MHD}%
If $E\into D\into H$ then $\MM(H;E)\subseteq\MM(H;D)$.
\end{enumerate}
\end{Rem}

Next we study the behavior of $D$-objects under induction and restriction.
\begin{Prop}
\label{Prop:D-obj-i_*}%
Let $D\ointo{j} H\ointo{i} G$. Then induction $i_*\colon \MM(H)\to \MM(G)$ preserves $D$-objects: $i_*(\MM(H;j))\subseteq \MM(G;ij)$, that is, $i_*(\MM(H;D))\subseteq \MM(G;D)$.
\end{Prop}

\begin{proof}
Direct from the definition.
\end{proof}

\begin{Prop}
\label{Prop:D-obj-i^*}%
Let $E\ointo{k} G$ and $H\ointo{i} G$. Then restriction $i^*\colon \MM(G)\to \MM(H)$ maps $E$-objects to $\isoc HEG$-objects, where $\isoc HEG$ embeds into~$H$ via the first projection $\pr_1\colon \Isoc HEG\into H$. Namely, if we have the left-hand isocomma
\[
\vcenter{\xymatrix@C=14pt@R=14pt{
& \isoc HEG \ar[dl]_-{\pr_1} \ar[dr]^-{\pr_2}
\\
H \ar[dr]_-{i} \ar@{}[rr]|-{\isocell{\gamma}}
&& E \ar[dl]^-{k}
\\
&G
}}
\qquadtext{then}
i^*(\MM(G;E))\subseteq\MM(H;\isoc{H}{E}{G})
\]
or more precisely $i^*(\MM(G;k))\subseteq\MM(H;\pr_1)$.
\end{Prop}

\begin{proof}
This follows immediately from the Mackey formula: $i^*k_*\cong (\pr_1)_*\pr_2^*$.
\end{proof}

\begin{Rem}
We are going to consider categories $\MM(G;E)$ and $\MM(H;E)$, with respect to isocommas~$E=\Isoc HDG$ for $D\into G$, or with respect to subgroupoids of such isocommas like $E=\partial_i(H,D)$ or $E=\partial_i(D,D)$ as in \Cref{sec:partial}. When the chosen faithful morphism $E\into H$ is not specified, we always use the first projection as in~\Cref{Conv:partial}. This is in line with \Cref{Prop:D-obj-i^*}.
\end{Rem}

The following will be a useful tool in the Green equivalence.
\begin{Lem}
\label{Lem:factor}%
Let $j\colon D\into H$ and $k\colon E\into H$. Let $\alpha\colon x\to y$ be a morphism in~$\MM(H)$, where $x$ is a $D$-object. Suppose that $\alpha$ factors via an $E$-object, then $\alpha$ factors via a $\isoc DEH$-object.
\end{Lem}

\begin{proof}
One easily retracts to the case where $x=j_*(w)$ is a strict $D$-object, for $w\in \MM(D)$, and where $\alpha$ factors via some strict $E$-object $k_*(z)$, with $z\in \MM(E)$. Let us write the assumed factorization, say, $\alpha=\gamma\beta$, in the top triangle below:
\[
\xymatrix@R=2em{
j_*(w) \ar[rr]^-{\alpha} \ar[rd]^-{\beta} \ar[d]_-{j_*(\tilde \beta)}
&& y
\\
j_*j^*k_*(z) \ar[r]_-{\eps}
& k_*(z) \ar[ru]^-{\gamma}
}
\]
Using the adjunction $\MM(H)(j_*w,k_*z)\simeq \MM(D)(w,j^*k_*z)$, there exists a morphism $\tilde\beta\colon w\to j^*k_* z$ making the above left-hand triangle commute, where $\eps$ is the counit of the $j_*\adj j^*$ adjunction. By Mackey for the isocomma~$\isoc DEH$, this new object $j_*j^*k_*(z)\simeq j_*{\pr_1}_*\pr_2^*(z)= (j \pr_1)_*\pr_2^*(z)$ is induced from $\pr_2^*(z)\in \MM(\isoc DEH)$ along $j \pr_1\colon \Isoc DEH\into H$, hence is a $\isoc DEH$-object, through which $\alpha$ factors.
\end{proof}

In particular, the $D$'s for which a fixed $x\in \MM(H)$ is a $D$-object are `filtering':

\begin{Cor} \label{Cor:filtered}
If $x\in \MM(H)$ is both a $D$-object and an $E$-object, then it is also a $\isoc{D}{E}{H}$-object (for either embedding of $\isoc{D}{E}{H}$ into~$H$).
\end{Cor}

\begin{proof}
If $x\in \MM(H;D)$ is also an $E$-object, then $\alpha=\id_x$ factors through an $E$-object, hence through a $\isoc{D}{E}{H}$-object by \Cref{Lem:factor}, hence $x\in \MM(H;\isoc{D}{E}{H})$.
\end{proof}

We shall now consider quotients $\frac{\MM(H)}{\MM(H;D)}$ of~$\MM(H)$ by~$D$-objects, as additive categories. See \Cref{Rec:add-quotient} if necessary.
\begin{Lem}
\label{Lem:ff-quotient}%
Let $D\overset{j}{\into} H$ and $E\overset{k}{\into} H$ and consider $\Isoc DEH\into H$ (either way). Then we have the right-hand commutative diagram of inclusions below
\[
\vcenter{\xymatrix@C=14pt@R=14pt{
& \isoc DEH \ar[dl]_-{\pr_1} \ar[dr]^-{\pr_2}
\\
D \ar[dr]_-{j} \ar@{}[rr]|-{\isoEcell{}}
&& E \ar[dl]^-{k}
\\
& H
}}
\qquad\leadsto\qquad\quad
\vcenter{\xymatrix@C=1pt@R=12pt{
& \MM(H;\isoc DEH) \kern-2em \ar@{^(->}[dl] \ar@{_(->}[dr]^-{}
\\
\kern-2em \MM(H;D) \ar@{_(->}[dr]_-{}
&& \MM(H;E) \ar@{^(->}[dl]^-{}
\\
&\MM(H)
}}
\]
Furthermore, the resulting canonical functor on the quotients
\begin{equation}
\label{eq:ff-quotient}%
\frac{\MM(H;D)}{\MM(H;\isoc DEH)}\too\frac{\MM(H)}{\MM(H;E)}
\end{equation}
is fully faithful.
\end{Lem}

\begin{proof}
The diagram of inclusions is immediate from $\MM(H;j \pr_1)=\MM(H;k \pr_2)=:\MM(H;\isoc DEH)$. Since both the inclusion $\MM(H;D)\hook \MM(H)$ and the projection $\MM(H)\onto \MM(H)/\MM(H;E)$ are full, the composite $\MM(H;D)\to \MM(H)/\MM(H;E)$ is certainly full and so is the induced functor~\eqref{eq:ff-quotient}. We need to show that the latter is faithful, which follows immediately from \Cref{Lem:factor}.
\end{proof}

\begin{Rem}
\label{Rem:disambig-D-objects}%
In \Cref{Lem:ff-quotient}, we saw quotients $\MM(H;D)/\MM(H;D')$ appear for the first time, with $D'\ointo{k} D\ointo{j} H$. Let us lift a possible ambiguity. By definition, this quotient consists of $D$-objects in~$\MM(H)$, with maps modulo $D'$-objects. It is however also the category of $D$-objects in~$\MM(H)/\MM(H;D')$, for $\MM(-)/\MM(-;D')$, that is, retracts of images of $j_*\colon \MM(D)/\MM(D;D')\to\MM(H)/\MM(H;D')$, which is well-defined by \Cref{Prop:D-obj-i_*}. Indeed, by \Cref{Rem:retract-in-quotient}, if $x\in \MM(H)$ is a retract in the quotient $\MM(H)/\MM(H;D')$ of some $j_*(y)$ for $y\in\MM(D)$ then $x$ is a retract of $j_*(y)\oplus (jk)_*(z)$ in~$\MM(H)$ for some $z\in\MM(D')$, and therefore $x$ is a retract of $j_*(y\oplus k_*(z))\in \MM(H;D)$. In other words, $\MM(H;D)/\MM(H;D')$ consists of the retract-closure in~$\MM(H)/\MM(H;D')$ of objects of the form $j_*(y)$ for $y\in\MM(D)$.
\end{Rem}

With \Cref{Prop:BN}, we get an easy condition for idempotent-completeness.

\begin{Prop}
\label{Prop:idempotent-complete}%
Let $\MM\colon \GG^\op\to \ADD$ be a Mackey 2-functor satisfying the following hypothesis:
\begin{enumerate}
  \item[$(\sqcup\Delta)$] For each $H\in \GG$, the category $\MM(H)$ admits countable coproducts and a triangulation such that $i^*\colon \MM(H)\to \MM(K)$ is exact for every $i\colon K\into H$.
\end{enumerate}
Then for every triple $D'\into D\into H$, the additive quotient $\MM(H;D)/\MM(H;D')$ is idempotent-complete.
\end{Prop}

\begin{proof}
As explained in \Cref{Rem:disambig-D-objects}, $\MM(H;D)/\MM(H;D')$ is closed under retracts in the ambient category~$\MM(H)/\MM(H;D')$. So it suffices to prove that the latter is idempotent-complete. We claim that it is triangulated and admits countable coproducts, so the result follows from \Cref{Prop:BN}. To get the triangulation, we can apply \Cref{Prop:triangulation-T/S} with $\cat{T}=\MM(H)$ and $\cat{S}=\MM(D')$ and the functors $R=\ell^*$ and $I=\ell_*$ for $\ell\colon D'\into H$. For coproducts, since $\MM(H)$ admits them, it suffices to show that $\MM(H;D')$ is closed under arbitrary coproducts. This follows from the fact that~$I=\ell_*$ commutes with coproducts, like every left adjoint.
\end{proof}

Returning to the functor~\eqref{eq:ff-quotient}, we now deduce from \Cref{Lem:ff-quotient} a statement which has the flavor of a `Second Isomorphism Theorem':

\begin{Cor}
\label{Cor:second-iso-thm}%
Let $j\colon D\into H$ and $k\colon E\into H$ and consider $(j, k)\colon D\sqcup E\into H$ as well as the isocomma~$\Isoc DEH\into H$ (either way). Then the canonical functor
\begin{equation}
\label{eq:second-iso-thm}%
\frac{\MM(H;D)}{\MM(H;\isoc DEH)}\too\frac{\MM(H;D\sqcup E)}{\MM(H;E)}
\end{equation}
given by the identity on objects and morphisms, is fully-faithful and induces an equivalence on idempotent-completions, \ie it is an equivalence-up-to-retracts.
\end{Cor}

\begin{proof}
We have $\MM(H;E)\subseteq\MM(H;D\sqcup E)\subseteq\MM(H)$. So, in the following commutative diagram of canonical functors, the right-hand vertical functor is fully faithful
\[
\xymatrix@C=4em@R=1em{
\Displ\frac{\MM(H;D)}{\MM(H;\isoc DEH)} \ar[r]^-{\textrm{\eqref{eq:second-iso-thm}}} \ar[rd]_-{\textrm{\eqref{eq:ff-quotient}}}
& \Displ \frac{\MM(H;D\sqcup E)}{\MM(H;E)} {}_\vcorrect{1.5} \ar@{_(->}[d]
\\
& \Displ \frac{\MM(H)}{\MM(H;E)}\,.
}
\]
As the diagonal functor~\eqref{eq:ff-quotient} is fully faithful by \Cref{Lem:ff-quotient}, so is our horizontal functor~\eqref{eq:second-iso-thm}. Let now $m\in \MM(H;D\sqcup E)$ and choose $(x,y)\in \MM(D)\oplus\MM(E)=\MM(D\sqcup E)$ such that $m\le j_*(x)\oplus k_*(y)$. Since $k_*(y)=0$ in the quotient modulo~$\MM(H;E)$, we see that $m$ is a retract of the strict $D$-object $j_*(x)$ in $\frac{\MM(H;D\sqcup E)}{\MM(H;E)}$. As $j_*(x)\in \MM(H;D)$, this shows that~\eqref{eq:second-iso-thm} is surjective-up-to-retracts.
\end{proof}

\begin{Exa}
When $D\ointo{j} H\ointo{i} G$ we have seen in Example~\ref{Exas:partial-HD}\,\eqref{Exa:partial-HD} a decomposition to which we can apply \Cref{Cor:second-iso-thm}, namely $\Isoc HDG \cong D\sqcup \partial_i(H,D)$.
\end{Exa}

\begin{Prop}
\label{Prop:proto-Green}%
Let $D\ointo{j} H\ointo{i} G$. The canonical functor (see \Cref{Conv:partial})
\[
\frac{\MM(H;D)}{\MM(H;\partial_i(D,D))}\too \frac{\MM(H;\isoc HDG)}{\MM(H;\partial_i(H,D))}
\]
is an equivalence-up-to-retracts.
\end{Prop}

\begin{proof}
Use $\Isoc HDG\cong D\sqcup \partial_i(H,D)$ by Example~\ref{Exas:partial-HD}\,\eqref{Exa:partial-HD} and apply \Cref{Cor:second-iso-thm} with $E=\partial_i(H,D)$. To apply \Cref{Cor:second-iso-thm}, we need to compute $\Isoc DEH$. But \Cref{Lem:geography} precisely gives us $\Isoc DEH=\Isoc D{\partial_i(H,D)}H\cong \partial_i(D,D)$.
\end{proof}

\begin{Thm}[The Green equivalence]
\label{Thm:Green-equivalence}%
Let $D\ointo{j} H\ointo{i} G$ be faithful morphisms of groupoids in~$\GG$ and let $\MM$ be a Mackey 2-functor. Then the induction functor $i_*\colon \MM(H)\to \MM(G)$ yields a well-defined equivalence of categories on the following idempotent-completions (\Cref{Def:idempotent-completion}) of additive quotients (\Cref{Rec:add-quotient})
\begin{equation}
\label{eq:Green-i_*}%
i_*\colon \bigg(\frac{\MM(H;D)}{\MM(H;\partial_i(D,D))}\bigg)^\natural \isotoo \bigg(\frac{\MM(G;D)}{\MM(G;\partial_i(D,D))}\bigg)^\natural.
\end{equation}
Moreover, the restriction functor $i^*\colon \MM(G)\to \MM(H)$ maps $\MM(G;\partial_i(D,D))$ into $\MM(H;\partial_i(H,D))$ and defines an equivalence
\begin{equation}
\label{eq:Green-i^*}%
i^*\colon \bigg(\frac{\MM(G;D)}{\MM(G;\partial_i(D,D))}\bigg)^\natural \isotoo \bigg(\frac{\MM(H;\isoc HDG)}{\MM(H;\partial_i(H,D))}\bigg)^\natural.
\end{equation}
Finally, the composite of those two equivalences is isomorphic to the canonical functor $\big(\frac{\MM(H;D)}{\MM(H;\partial_i(D,D))}\big)^\natural\isoto\big(\frac{\MM(H;\isoc HDG)}{\MM(H;\partial_i(H,D))}\big)^\natural$ of \Cref{Prop:proto-Green}.
\end{Thm}

\begin{proof}
By \Cref{Prop:D-obj-i_*}, we know that $i_*\colon \MM(H)\to \MM(G)$ restricts to a functor $\MM(H;E)\to \MM(G;E)$ for all subgroupoids~$E\into H$ that we care about, namely $D$ and $\partial_i(D,D)$. So the induction functor~\eqref{eq:Green-i_*} is well-defined. With restriction we need to be a little more careful. The general \Cref{Prop:D-obj-i^*} only tells us that $i^*\colon \MM(G)\to \MM(H)$ restricts to $\MM(G;D)\to \MM(H;\isoc HDG)$ and further to $\MM(G;\partial_i(D,D))\to \MM(H;\isoc H{\partial_i(D,D)}G)$. So the `numerators' in~\eqref{eq:Green-i^*} are fine but for the `denominators' we still need to prove the inclusion $\MM(H;\isoc H{\partial_i(D,D)}G)\subseteq\MM(H,\partial_i(H,D))$. For this it suffices that the morphism $\isoc H{\partial_i(D,D)}G\into H$ (given by $\pr_1$) factors via $\pr_1\colon \partial_i(H,D)\into H$. This is precisely the content of \Cref{Lem:tricky}. Therefore all functors in the following diagram are well-defined:
\begin{equation}
\label{eq:aux-Green}%
\vcenter{\xymatrix@C=5em@R=2em{
\Displ\frac{\MM(H;D)}{\MM(H;\partial_i(D,D))} \ar[d]^-{i_*} \ar@/^1em/[rd]^-{\textrm{(\ref{Prop:proto-Green})}}_-{\natural\textrm{-}\cong}
\\
\Displ\frac{\MM(G;D)}{\MM(G;\partial_i(D,D))} \ar[r]^-{i^*}
& \Displ\frac{\MM(H;\isoc HDG)}{\MM(H;\partial_i(H,D))}
}}
\end{equation}
We first claim that this diagram commutes up to isomorphism. Indeed, recall from~\cite[(Mack\,9) in Theorem~1.2.1]{BalmerDellAmbrogio18pp} that the unit $\eta\colon \Id_{\MM(H)}\Rightarrow i^*i_*\colon \MM(H)\to \MM(H)$ of the $i^*\dashv i_*$ adjunction is a naturally split monomorphism. More precisely, if we decompose the self-isocomma $\isoc HHG$ as in Example~\ref{Exas:partial-HD}\,\eqref{Exa:partial-HH}
\begin{equation*}
%\label{eq:aux-ii}%
\vcenter{
\xymatrix@C=14pt@R=14pt{
& \kern-1em H\sqcup \partial_i(H,H) \kern-1em \ar[ld]_(.6){(\Id_H\,,\, p_1)} \ar[rd]^(.6){(\Id_H\,,\, p_2)} \ar@{}[dd]|{\isocell{(\id_i \,,\, \gamma')}}
\\
H \ar[rd]_-{i}
&& H \ar[ld]^-{i}
\\
& G
}}
\end{equation*}
then $\smat{\eta & \gamma'_!}\colon \Id_{\MM(H)}\oplus {p_1}_*p_2^*\isoTo i^*i_*$. (We temporarily denote the projections restricted to $\partial_i(H,H)$ by $p_1$ and $p_2$ to avoid confusion.) So the `difference' between $i^*i_*$ and $\Id_{\MM(H)}$ is given by~${p_1}_*p_2^*$. When evaluating ${p_1}_*p_2^*$ at a strict $D$-object in~$\MM(H)$, say, $x=j_*(y)$ for some $y\in \MM(D)$, we need to use a diagram of Mackey squares as in \Cref{Lem:geography} (see the bottom-right of~\eqref{eq:geography} for $D_2=D$):
\begin{equation}
\label{eq:bottom-right-big-diag}%
\vcenter{
\xymatrix@C=14pt@R=14pt{
&& \kern-1em D\sqcup \partial_i(H,D) \kern-1em \ar[ld]_(.6){j \,\sqcup\, \partial_i(H,j) \;\;} \ar[rd]^(.6){(\Id_D\,,\, \pr_2)}
\\
& \kern-1em H\sqcup \partial_i(H,H) \kern-1em \ar[ld]_(.6){(\Id_H\,,\, p_1)} \ar[rd]^(.6){\; \; (\Id_H\,,\, p_2)} \ar@{}[dd]|{\isocell{(\id_i\,,\, \gamma')}}
&& D \ar[ld]^-{j}
\\
H \ar[rd]_-{i}
&& H \ar[ld]^-{i}
\\
& G
}}
\end{equation}
We can then compute
\[
{p_1}_*p_2^*j_*(y)\simeq {p_1}_*(\partial_i(H,j))_*\pr_2^*(y)\simeq (p_1\circ\partial_i(H,j))_*\pr_2^*(y).
\]
Since $p_1\circ \partial_i(H,j)=\pr_1\colon \partial_i(H,D)\into H$, we see that the above `difference' object ${p_1}_*p_2^*j_*(y)$ between $x=j_*(y)$ and $i^*i_*(x)\simeq x\oplus {p_1}_*p_2^*j_*(y)$ is actually zero in the quotient by $\MM(H,\partial_i(H,D))$. This proves that $\eta$ induces an isomorphism making~\eqref{eq:aux-Green} commute, $\Id\isoTo i^*i_*$, as claimed.

We then claim that $i^*$ in~\eqref{eq:aux-Green} is faithful. Let $x,y\in \MM(G;D)$ be $D$-objects over~$G$ and $\alpha\colon x\to y$ such that $i^*(\alpha)$ factors via a (strict) $\partial_i(H,D)$-object in~$\MM(H)$. We need to show that $[\alpha]=0$ in $\MM(G)/\MM(G;\partial_i(D,D))$. Up to retraction, we can assume that $x$ and $y$ are strict $D$-objects and in particular $x=i_*(w)$ for a (strict) $D$-object $w\in \MM(H;D)$ over~$H$. So we have $\alpha\colon x=i_*(w)\to y$. By adjunction, there exits a map $\beta\colon w\to i^*(y)$ in $\MM(H)$ adjoint to~$\alpha$, \ie such that
\begin{equation}
\label{eq:aux-alpha-beta}%
\alpha=\eps_y\circ i_*(\beta)\colon \quad x=i_*(w)\xto{i_*(\beta)} i_*i^*(y) \oto{\eps} y
\end{equation}
in~$\MM(G)$. The inverse formula for $\beta$ is that $\beta=i^*(\alpha)\circ\eta_w$ and in particular, since $i^*(\alpha)$ is assumed to factor via a $\partial_i(H,D)$-object then so does $\beta$. But $\beta\colon w\to i^*(y)$ starts from a (strict) $D$-object. So by \Cref{Lem:factor} for $E=\partial_i(H,D)$, we know that $\beta$ also factors via a $\isoc D{\partial_i(H,D)}H$-object. By \Cref{Lem:geography} (for $D_1=D_2=D$) we know that $\isoc D{\partial_i(H,D)}H\cong\partial_i(D,D)$. So $\beta$ factors over~$H$ via a strict  $\partial_i(D,D)$-object, say $(j\pr_1)_*(z)$ and therefore $i_*(\beta)$ factors via $(ij\pr_1)_*(z)$, which is now a $\partial_i(D,D)$-object over~$G$. Combining this with~\eqref{eq:aux-alpha-beta}, we see that $\alpha$ factors via the same object, that is, $[\alpha]=0$ in $\MM(G)/\MM(G;\partial_i(D,D))$, as claimed.

Returning to~\eqref{eq:aux-Green}, we now know that $i^*\circ i_*$ is an equivalence-up-to-retracts and that $i^*$ is faithful. From the former, it follows that $i^*$ is surjective-up-to-retracts and that $i^*$ is full at least on images of~$i_*$. But every object of~$\MM(G;D)/\MM(G;\partial_i(D,D))$ is a retract of such an image under~$i_*$ hence $i^*$ is indeed full. In short, $i^*$ is an equivalence-up-to-retracts and therefore so is~$i_*$.
\end{proof}

%------------------------------------------------------------------------------
%
\section{The Krull-Schmidt case}
\label{sec:Green-corr}%
%\bigbreak
%------------------------------------------------------------------------------

We now specialize the Green equivalence of \Cref{sec:Green} to the classical setting.
\begin{Hyp}
\label{Hyp:KS}%
In this section, $\MM\colon \GG^\op\to \ADD$ will be a Mackey 2-functor such that each category~$\MM(G)$ is \emph{Krull-Schmidt} (see \Cref{Rec:KS}).
\end{Hyp}

In this situation it becomes possible to explicitly describe a quasi-inverse to the Green equivalence in terms of the restriction functor; see \Cref{Cor:Green-eq-KS}. The \emph{Green correspondence} for indecomposable objects can then be deduced from it; see \Cref{Cor:Green-corr}. Specializing to groups as in \Cref{Rem:groupist-translation}, we recover Green's original results \cite{Green72} and \cite{Green64}  in modular representation theory; see \Cref{Exa:Green-corr-modrep}.

We did not use Krull-Schmidt so far. Let us remind the reader.

\begin{Rec}
\label{Rec:KS}%
An additive category $\cat A$ is \emph{Krull-Schmidt} if every object $x\in\cat{A}$ admits a decomposition $x\simeq x_1\oplus \ldots \oplus x_r$ where all $x_i$ have \emph{local} endomorphism rings. (A nonzero ring is local if its non-invertible elements are closed under addition.) An object with local endomorphism ring is indecomposable. Then $\cat{A}$ has the \emph{Krull-Schmidt decomposition property}: For two decompositions as above $x_1\oplus \ldots \oplus x_r\simeq x'_1\oplus \ldots \oplus x'_s$ we have $r=s$ and isomorphisms $x_i \simeq y_{\sigma(i)}$ for a permutation $\sigma \in \Sigma_r$.
By \cite[Cor.\,4.4]{Krause15}, the category $\cat{A}$ is Krull-Schmidt if and only if it is idempotent-complete and $\End_{\cat{A}}(x)$ is semi-perfect for all~$x\in\cat{A}$. As finite-dimensional algebras are semi-perfect, an idempotent-complete $\cat A$ is Krull-Schmidt if it is Hom-finite over a field~$\kk$ (\ie $\kk$-linear and $\dim_{\kk}(\Hom(x,y))<\infty$ for all $x,y\in\cat{A}$).
\end{Rec}

\begin{Rem}
\label{Rem:KS-quotients}%
Let $\cat{A}$ be Krull-Schmidt and let $\cat{B}\subseteq\cat{A}$ be a full additive subcategory closed under retracts. Then $\cat{B}$ is clearly Krull-Schmidt. The quotient $\cat A/\cat B$ is also Krull-Schmidt. Indeed, the functor $\cat A\to \cat A/\cat B$ sends finite direct sums to (possibly shorter) finite direct sums; it also sends an object with local endomorphism ring in $\cat A$ either to zero or again to such an object of~$\cat A/\cat B$, because $\End_{\cat A/\cat B}(x)$ is a quotient ring of $\End_{\cat A}(x)$. In particular, under \Cref{Hyp:KS}, the categories $\MM(G;D)$ and $\MM(G;D)/\MM(G;E)$ we considered in \Cref{sec:Green} are all Krull-Schmidt. Furthermore:
\end{Rem}

\begin{Lem}
\label{Lem:quot-indec}
Let $\cat{B}\subseteq\cat{A}$ be an additive subcategory closed under retracts. The quotient $\cat A\to \cat A/\cat B$ gives a bijection between the isomorphism classes of indecomposable objects in $\cat A$ not belonging to~$\cat B$ and those of~$\cat A/\cat B$. Furthermore, if $x$ has no summand in~$\cat{B}$ and is indecomposable in~$\cat{A}/\cat{B}$ then $x$ is indecomposable in~$\cat{A}$.
\end{Lem}

\begin{proof}
By \Cref{Rem:KS-quotients}, it suffices to show that indecomposables $x,y\in\cat{A}\smallsetminus\cat{B}$ that become isomorphic in~$\cat{A}/\cat{B}$ are already isomorphic in~$\cat{A}$. \Cref{Rem:retract-in-quotient} implies that $x$ is a retract of $y\oplus b$ for some $b\in \cat B$. As $x\not\in\cat B$, Krull-Schmidt forces $x\simeq y$.
\end{proof}

\begin{Not}
\label{Not:E-free}%
Let $\MM\colon \GG^\op\to\ADD$ as in \Cref{Hyp:KS}. For every faithful morphism $k\colon E\into H$ and every object $m\in \MM(H)$, we \emph{choose} a decomposition
\begin{equation}
\label{eq:KS-decomp}%
\varphi_m\colon m \isoto p'_E(m) \oplus p_E(m)
\end{equation}
with $p_E(m)\in \MM(H;E)$ and with $p'_E(m)$ having no indecomposable summand in~$\MM(H;E)$. This decomposition exists by simply regrouping indecomposable summands of~$m$, depending on whether they belong to~$\MM(H;E)$. We call $p_E(m)$ the \emph{$E$-part} of~$m$ and $p'_E(m)$ the \emph{$E$-free part} of~$m$.
For any map $\alpha \colon m\to n$ we write
\begin{equation}
\label{eq:KS-decomp-maps}%
p'_E(\alpha)\colon p'_E(m)\to p'_E(n)
\end{equation}
for the diagonal component of $\alpha$ on the $E$-free parts of its source and target.
\end{Not}

\begin{Rem}
\label{Rem:E-free}%
Of course, the above choice~\eqref{eq:KS-decomp} is non-canonical and \eqref{eq:KS-decomp-maps} is not functorial. However, we are going to see that this issue disappears on suitable additive subquotients.
\end{Rem}

\begin{Prop}
\label{Prop:second-iso-thm-inverse}%
Consider the canonical equivalence of \Cref{Cor:second-iso-thm}
\begin{equation}
\label{eq:second-iso-thm-repeated}%
\frac{\MM(H;D)}{\MM(H;\isoc DEH)}\isotoo\frac{\MM(H;D\sqcup E)}{\MM(H;E)}
\end{equation}
For $j\colon D\into H$ and $k\colon E\into H$. Then the $E$-free part assignments $m \mapsto p'_E(m)$ and $\alpha \mapsto p'_E(\alpha)$ of~\Cref{Not:E-free} induce a well-defined quasi-inverse to~\eqref{eq:second-iso-thm-repeated}
\begin{equation}
\label{eq:second-iso-thm-inverse}%
p'_E\colon \frac{\MM(H;D\sqcup E)}{\MM(H;E)} \isotoo \frac{\MM(H;D)}{\MM(H;\isoc DEH)}.
\end{equation}
\end{Prop}

\begin{proof}
Since Krull-Schmidt implies idempotent-complete, \eqref{eq:second-iso-thm-repeated} follows from~\eqref{eq:second-iso-thm}. We claim that our chosen $p'_E$ as in \Cref{Not:E-free} yields a well-defined functor
\begin{equation}
\label{eq:pre-eq-E-free}%
p'_E\colon \MM(H;D\sqcup E) \too \frac{\MM(H;D)}{\MM(H; \isoc DEH)} \,.
\end{equation}
Let $m\in \MM(H;D\sqcup E)$, say $m\le j_*(v)\oplus k_*(w)$. For every indecomposable $x\le m$, we must have $x\le j_*(v)\in\MM(H;D)$ or $x\le k_*(w)\in\MM(H;E)$. By definition of the $E$-free part of~$m$, we get $p'_E(m)\in \MM(H;D)$. So \eqref{eq:pre-eq-E-free} makes sense on objects and we define it as $\alpha\mapsto [p'_E(\alpha)]$ on maps. Identity maps are sent to identities. For a composite $m \xrightarrow{\alpha} n \xrightarrow{\beta} q$ in $\MM(H;D\sqcup E)$, we have a commutative diagram
\[
\xymatrix@C=6em@R=2em{
m \ar[d]_-{\varphi_m}^-{\simeq} \ar[r]^-{\alpha}
& n \ar[d]_-{\varphi_n}^-{\simeq} \ar[r]^-{\beta}
& q \ar[d]_-{\varphi_q}^-{\simeq}
\\
p'_E(m)\oplus p_E(m) \ar[r]^-{\smat{p'_E(\alpha) & a_1\\a_2& a_3}}
&
p'_E(n)\oplus p_E(n) \ar[r]^-{\smat{p'_E(\beta) & b_1\\b_2& b_3}}
&
p'_E(q)\oplus p_E(q)
}
\]
and therefore $p'_E(\beta\alpha)=p'_E(\beta)p'_E(\alpha)+b_1 a_2$. The `error' $b_1 a_2$ factors via $p_E(n)$ and its source $p'_E(m)$ is a $D$-object by the above discussion. By \Cref{Lem:factor}, this error $b_1a_2$ factors via a $\isoc DEH$-object, hence disappears in $\frac{\MM(H;D)}{\MM(H; \isoc DEH)}$. So~\eqref{eq:pre-eq-E-free} is a well-defined functor. \Cref{Lem:factor} also guarantees that the functor~\eqref{eq:pre-eq-E-free} descends to a well-defined functor $p'_E$ as in~\eqref{eq:second-iso-thm-inverse}.
It is then easy to see that this $p'_E$ followed by the canonical equivalence~\eqref{eq:second-iso-thm-repeated} is isomorphic to the identity. The isomorphism is given by $[\varphi_m]\colon m\isoto p'_E(m)\oplus p_E(m)\cong p'_E(m)$ where the latter projection is an isomorphism modulo~$\MM(H;E)$ since $p_E(m)\in\MM(H;E)$. The naturality of this transformation~$[\varphi_m]$ in~$m$ comes from the very definition of~$p'_E(\alpha)$.
\end{proof}

\begin{Cor}[The Green equivalence, Krull-Schmidt case] \label{Cor:Green-eq-KS}
Let $D\ointo{j} H\ointo{i} G$ and let $\MM$ be a Mackey 2-functor taking values in Krull-Schmidt categories.
Then $i_*\colon \MM(H)\to \MM(G)$ induces a well-defined equivalence of the subquotient categories
\begin{equation*}
i_*\colon \frac{\MM(H;D)}{\MM(H;\partial_i(D,D))} \isotoo \frac{\MM(G;D)}{\MM(G;\partial_i(D,D))}
\end{equation*}
isomorphic to $p'_{\partial_i(D,D)}\circ i_*$, namely the $\partial_i(D,D)$-free part (as in \Cref{Not:E-free}) of induction. Its quasi-inverse is given by the $\partial_i(H,D)$-free part of restriction~$i^*$
\[
p'_{\partial_i(H,D)} \circ i^* \colon  \frac{\MM(G;D)}{\MM(G;\partial_i(D,D))} \isotoo \frac{\MM(H;D)}{\MM(H;\partial_i(D,D))}  \,.
\]
\end{Cor}

\begin{proof}
We have three equivalences, by \eqref{eq:Green-i_*}, \eqref{eq:Green-i^*} and~\eqref{eq:second-iso-thm-inverse}
\[
\xymatrix@C=1.5em{
\Displ \frac{\MM(H;D)}{\MM(H;\partial_i(D,D))} \ar[r]^-{i_*}_-{\cong}
& \Displ \frac{\MM(G;D)}{\MM(G;\partial_i(D,D))} \ar[r]^-{i^*}_-{\cong}
& \Displ \frac{\MM(H;\isoc{H}{D}{G})}{\MM(H;\partial_i(H,D))} \ar[r]^-{p'_E}_-{\cong}
& \Displ \frac{\MM(H;D)}{\MM(H;\partial_i(D,D))}
}
\]
where the latter groupoid $E$ is~$\partial_i(H,D)$, using that $\isoc HDG\cong D\sqcup\partial_i(H,D)$ and $\isoc{D}{\partial_i(H,D)}{H}\cong \partial_i(D,D)$ together with \Cref{Prop:second-iso-thm-inverse}. That proposition also guarantees that the composition of these three equivalences is isomorphic to the identity. This shows that $p'_E\circ i^*$ is a quasi-inverse to~$i_*$. Also the functor~$i_*$ is clearly isomorphic to $p'_{\partial_i(D,D)}\circ i_*$, since $p_{\partial_i(D,D)}(i_*(n))\cong 0$ modulo~$\MM(G;\partial_i(D,D))$.
\end{proof}

We deduce the usual result for indecomposable objects `upstairs'.

\begin{Cor}[The Green correspondence]
\label{Cor:Green-corr}%
Let $D\ointo{j} H\ointo{i} G$ be faithful functors of finite groupoids, and let $\MM$ be a Mackey 2-functor taking values in Krull-Schmidt categories.
Then there exists a bijection between the isomorphism classes of indecomposable objects $m\in \MM(G;D)$ not belonging to~$\MM(G;\partial_i(D,D))$ and the isomorphism classes of indecomposable objects $n \in \MM(H;D)$ not belonging to~$\MM(H;\partial_i(D,D))$. It maps an indecomposable $m\in \MM(G;D)$ to the $\partial_i(H,D)$-free part $p'_{\partial_i(H,D)} (i^*(m))$ of its restriction and an idecomposable $n$ to the $\partial_i(D,D)$-free part $p'_{\partial_i(D,D)} (i_*(n))$ of its induction.
Alternatively: $m$ and $n$ correspond to each other if and only if either, and hence both, of the following two conditions hold:
\[
m\leq i_*(n) \quad \textrm{ and } \quad n \leq i^*(m)\,.
\]
\end{Cor}

\begin{proof}
Chase indecomposable objects under the three bijections given by:
\begin{enumerate}[(1)]
\item
\Cref{Lem:quot-indec} for $\cat{A}=\MM(H;D)$ and $\cat{B}=\MM(H;\partial_i(D,D))$;
\item
The equivalence of~\eqref{Cor:Green-eq-KS};
\item
\Cref{Lem:quot-indec} for $\cat{A}=\MM(G;D)$ and $\cat{B}=\MM(G;\partial_i(D,D))$.
\end{enumerate}
The final reformulation follows easily from \Cref{Rem:retract-in-quotient} and \Cref{Lem:quot-indec} again.
\end{proof}

\begin{Rem} \label{Rem:normalizer-hyp}
The original version of the Green equivalence in \cite[Theorem~4.1]{Green72} is for the Mackey 2-functor of modular representations (see~\Cref{Exa:Green-corr-modrep}\,\eqref{it:mod} below) and is expressed in terms of finite groups $D\leq H \leq G$, as in \Cref{Rem:groupist-translation}. It also has the additional hypothesis that $H$ contains the normalizer~$N_G(D)$.
The latter guarantees that $D\cap {}^gD\lneq D$ as long as $g\notin H$, so that (as soon as $H\lneq G$)  the category
$
\MM(H;D)/\MM(H;\partial_i(D,D)) = \MM(H;D)/\MM(H; \sqcup_{[g], g\not\in H} D\cap{}^gD)
$
is non-zero, and similarly for~$G$, so that the Green equivalence is not an empty statement.
\end{Rem}

To conclude this section, still working under the Krull-Schmidt hypothesis, we further refine the Green correspondence by introducing vertices of objects.

\begin{Rem}[Vertices and sources] \label{Rem:vertices-and-sources}
Just as in modular representation theory, we can define the vertex and the source of an indecomposable object $m\in \MM(G)$ for any Krull-Schmidt Mackey 2-functor~$\MM$. First, by the additivity of~$\MM$ and the fact that $m$ is indecomposable, we may always arrange for $G$ to be a group.
Then we may define a \emph{vertex} of $m$ to be a subgroup $i\colon D\hookrightarrow G$ such that $m\in \MM(G;D)$ and which, among such subgroups of~$G$, is minimal with respect to inclusion. A \emph{source} of $m$ is then an indecomposable object $s$ of $\MM(G;D)$ such that $m\leq i_*(s)$.
One can prove, precisely as in the proof of \cite[Theorem~9.4]{Alperin80}, that every $m$ admits a vertex and a source, that any two vertices of $m$ are $G$-conjugate subgroups, and that any two sources of $m$ (for the same vertex) are isomorphic. (In fact, if $D_1,D_2$ are two vertices of $m$ and $s_1\in \MM(D_1)$ and $s_2\in \MM(D_2)$ two sources, then there exists an element $g\in G$ such that ${}^gD_1= D_2$ and $c^*_g(s_2)\simeq s_1$, where $c_g^*\colon \MM(S_2)\overset{\sim}{\to} \MM(S_1)$ is the isomorphism of categories induced by conjugation $c_g\colon D_1 \overset{\sim}{\to} D_2$, $x\mapsto {}^gx$.)
\end{Rem}

\begin{Prop} [Green correspondents have same vertex]
\label{Prop:vertices-match}%
Let $D\leq H \leq G$ be finite groups.
If the indecomposable objects $m\in \MM(G;D)\smallsetminus \MM(G;\partial_i(D,D))$ and $n\in \MM(H;D) \smallsetminus \MM(H;\partial_i(D,D))$ are matched by the Green correspondence  of \Cref{Cor:Green-corr}, then their vertices, defined as in \Cref{Rem:normalizer-hyp}, are $G$-conjugate (\ie are `the same'). This vertex is necessarily ($G$-conjugate to) a subgroup $P\leq D$ which is not $G$-subconjugate to any subgroup of the form $D\cap {}^gD$ with $g\in G\smallsetminus H$.
\end{Prop}

\begin{proof}
Let $i\colon H\rightarrowtail G$ be the inclusion. We are in the situation of \Cref{Rem:groupist-translation}, where there is an equivalence $\partial_i(D,D)\simeq \coprod_{[g]\in \doublequot DGD,\,g\notin H} D\cap {}^gD$. Hence the Green correspondence for $D\leq H \leq G$ concerns only those indecomposable objects in $\MM(G)$ and $\MM(H)$ which are not $D\cap {}^gD$-objects for any $g\in G\smallsetminus H$.

Let $m\in \MM(G)$ and $n\in \MM(H)$ be Green correspondents as in \Cref{Cor:Green-corr}:
\begin{equation} \label{eq:a-b-weak}
m \leq i_*(n)
\quad \textrm{ and } \quad
n \leq i^*(m).
\end{equation}
Let $P$ and $Q$ be vertices of $m$ and of~$n$, respectively.
In particular $n$ is a $Q$-object, hence so is the induced $i_*(n)$ and therefore so is its retract~$m$.
By the uniqueness of the vertex~$P$ of~$m$, this implies that a conjugate of~$P$ is a subgroup of~$Q$.

On the other hand, $n$ is a retract of $i^*(m)$ by \eqref{eq:a-b-weak}. And $m$ is a $P$-object (as $P$ is the vertex of~$m$) hence its restriction $i^*(m)$ is an $\isoc{H}{P}{G}$-object by \Cref{Prop:D-obj-i^*} (applied with $E:=P$).
Because of the equivalence $\Isoc{H}{P}{G} \simeq \coprod_{[g]\in \doublequot H G P} H \cap {}^{g\!}P$, the indecomposable $n$ is already an $H \cap {}^{g\!}P$-object for some $g\in G$.
By the uniqueness of the vertex~$Q$, this shows that a conjugate of $Q$ is a subgroup of~$P$.

It follows that $P$ and $Q$ are conjugate subgroups of~$G$.
\end{proof}

\begin{Cor}
\label{Cor:Green-corr-vertex}%
Let $D\leq H \leq G$ be finite groups such that $N_G(D)\leq H$.
Then the Green correspondence  of \Cref{Cor:Green-corr} restricts to a bijection between indecomposables in $\MM(G)$ with vertex~$D$ and indecomposables in $\MM(H)$ with vertex~$D$.
\end{Cor}

\begin{proof}
Since $H\subseteq N_G(D)$, for $g\in G\smallsetminus H$ the subgroup $D\cap {}^gD$ cannot contain a conjugate of~$D$ because its order is strictly smaller. Hence $D$ is an admissible vertex to which we may restrict the Green correspondence as in~\Cref{Prop:vertices-match}.
\end{proof}

We conclude with a conceptual explanation for why all vertices arising in modular representation theory are necessarily $p$-groups.

\begin{Rem}
\label{Rem:cohomological}%
We say that a (non-necessarily Krull-Schmidt) Mackey 2-functor~$\cat M$ is \emph{cohomological}\footnote{We study cohomological Mackey 2-functors, in the above sense, in a separate article in preparation, where in particular we will justify the terminology `cohomological'.} if for every inclusion $i\colon H\hookrightarrow G$ of a subgroup, the composite
\[
\Id_{\MM(G)}
\overset{\reta\;}{\Longrightarrow}
i_*i^*
\overset{\leps\;}{\Longrightarrow}
\Id_{\MM(G)}
\]
of the unit of the `right' adjunction $i^*\dashv i_*$ followed by the counit of the `left' adjunction $i_*\dashv i^*$, is equal to the natural transformation $[G:H]\cdot \id$.

In that case, for every subgroup $H\leq G$ such that $[G:H]$ is invertible in~$\MM(G)$ (\ie such that $\MM(G)$ is a $\bbZ[\frac{1}{[G:H]}]$-linear category) each $m\in \MM(G)$ is an $H$-object.
For instance, if~$\MM(G)$ is $\bbZ_{(p)}$-linear for some prime number~$p$ (\eg $\bbZ/p\bbZ$-linear), each object $m\in \MM(G)$ is an $H$-object for $H$ a $p$-Sylow subgroup of~$G$. In particular, in the Krull-Schmidt case, if $m$ is indecomposable then its vertex (as in \Cref{Rem:vertices-and-sources}) is a $p$-subgroup of~$G$.

In the extreme case where $\MM$ is $\mathbb Q$-linear and cohomological then $\cat M(G)=\cat M(G;1)$ for all~$G$, where 1 is the trivial group, and all vertices are trivial.
\end{Rem}

%------------------------------------------------------------------------------
%
\section{Examples in algebra, topology and geometry}
\label{sec:examples}%
%\bigbreak
%------------------------------------------------------------------------------

Since \Cref{Thm:Green-equivalence} can be applied to any Mackey 2-functor~$\MM$, we obtain a Green equivalence theorem for \emph{each} of the many examples of mentioned in \cite[Chapter~4]{BalmerDellAmbrogio18pp}. As illustration, let us detail here a few of the most noteworthy ones, pointing out the Krull-Schmidt cases in which the Green correspondence of \Cref{Cor:Green-corr} also holds, and starting with the classical setting.

\begin{Exa}[Modular representation theory] \label{Exa:Green-corr-modrep}
Let $\kk$ be a field of positive characteristic~$p>0$.
The typical `small' Mackey 2-functors used in modular representation theory over~$\kk$ are $\Hom$-finite over~$\kk$, idempotent-complete and therefore Krull-Schmidt by \Cref{Rec:KS}. So we obtain for them both the Green equivalence and the Green correspondence. These $\MM$ are given as follows, for a group~$G$:
\begin{enumerate} [\rm(a)]
\item
\label{it:mod}%
$\MM(G)= \fgmod(\kk G)$, the category of finite-dimensional $\kk$-linear representations of~$G$, \ie finitely generated left $\kk G$-modules,  and $\kk G$-linear maps.
\item
\label{it:stmod}%
$\MM(G)= \stmod(\kk G):= \frac{\fgmod(\kk G)}{\fgproj(\kk G)}$, the stable module category of~$\kk G$.
\item
\label{it:der}%
$\MM(G)= \Db(\kk G)$, the derived category of bounded complexes in~$\fgmod(\kk G)$.
\end{enumerate}
In each case, the induction and restriction adjunctions are given by the usual functors; see \cite[\S\,4.1-2]{BalmerDellAmbrogio18pp} for more details.
An easy direct calculation shows that these Mackey 2-functors are cohomological as in \Cref{Rem:cohomological}: Hence for them, the vertex of any indecomposable object must be a $p$-group.

This recovers the standard formulation of the Green correspondence, \cf \cite{Alperin80}.
\end{Exa}

\begin{Rem}
Consider the classical situation of finite groups $D \subseteq N_G(D)\subseteq H \subsetneq G$, where in particular $\partial_i(D,D)\neq \emptyset$. Recall that a modular representation is projective if and only if its vertex is the trivial group, hence we have $\fgproj (\kk G)\subseteq \fgmod(G;\partial_i(D,D))$, and similarly for~$H$. Thus the Green equivalence and correspondence for the Mackey 2-functors \eqref{it:mod} and \eqref{it:stmod} are identical. Those for~\eqref{it:der}, on the other hand, are interesting in their own right since there are many indecomposable complexes which are not merely (shifted) representations. Compare~\cite{CarlsonWangZhang20pp}.
\end{Rem}

\begin{Rem}
In \Cref{Exa:Green-corr-modrep} one may equally well take an arbitrary coefficient ring~$\kk$, at the risk of possibly losing the Krull-Schmidt property.
\end{Rem}

\begin{Rem}
The Green equivalence can also be applied to the `big' versions of the Mackey 2-functors of \Cref{Exa:Green-corr-modrep}:
\begin{enumerate} [\rm(a)]
\item
\label{it:Mod}%
$\MM(G)= \Mod(\kk G)$, the category of all (also infinitely-dimensional) $\kk G$-modules.
\item
\label{it:Stmod}%
$\MM(G)= \StMod(\kk G):= \frac{\Mod(\kk G)}{\Proj(\kk G)}$, the stable category of all $\kk G$-modules.
\item
\label{it:Der}%
$\MM(G)= \Der(\kk G)$, the derived category of all (also unbounded) complexes.
\end{enumerate}
Further variations are possible, for instance by considering the \emph{homotopy} category, or \emph{bounded} chain complexes, etc., as long as they still form a Mackey 2-functor.

This recovers the main results of \cite{BensonWheeler01} and \cite{CarlsonWangZhang20pp}.
\end{Rem}

\begin{Exa}[Stable homotopy]
\label{Exa:top}%
As in \cite[Example 4.3.8]{BalmerDellAmbrogio18pp}, we can consider the Mackey 2-functor where $\MM(G)= \SH(G)$ is the stable homotopy category of genuine $G$-equivariant spectra. Alternatively, as in \cite[Example 4.1.6]{BalmerDellAmbrogio18pp}, there is a Mackey 2-functor where $\MM(G)$ is the homotopy category of $G$-diagrams in spectra, a.k.a.\ the `very na\"ive' $G$-spectra.
Neither of these examples is Krull-Schmidt, not even if we restrict attention to compact objects (\ie finite spectra); indeed the `plain' category of finite spectra $\SH(1)^c=\Ho(\mathcal Sp^{f})$ is not Krull-Schmidt by \cite[Curiosity~6.3]{Freyd66}. (Note that the categories of socalled `na\"ive' equivariant spectra do not form a Mackey 2-functor, as they do not satisfy ambidexterity (Mack~4).)
\end{Exa}

\begin{proof}[Proof of \Cref{Cor:intro-top}]
Use \Cref{Thm:Green-equivalence} for the Mackey 2-functor $\MM(-)=\SH(-)$ of \Cref{Exa:top} and use \Cref{Prop:idempotent-complete} to drop idempotent-completion.
\end{proof}

\begin{Exa}[Kasparov theory]
\label{Exa:KK}%
As in \cite[Example 4.3.9]{BalmerDellAmbrogio18pp}, there is a Mackey 2-functor whose value $\MM(G)=\KK(G)$ is $G$-equivariant KK-theory, a.k.a.\ the Kasparov category of separable $G$-C*-algebras. Variations are possible: We may consider equivariant $E$-theory instead (see \cite{GuentnerHigsonTrout00}), or the Mackey 2-functor formed by the subcategories of $G$-cell algebras in the sense of \cite{DellAmbrogio14}, or the subcategories of compact objects therein.
We do not know if the latter example is Krull-Schmidt, but  in analogy to \Cref{Exa:top} we rather expect not.
\end{Exa}

\begin{proof}[Proof of \Cref{Cor:intro-KK}]
Use \Cref{Thm:Green-equivalence} for the Mackey 2-functor $\MM(-)=\KK(-)$ of \Cref{Exa:KK} and use \Cref{Prop:idempotent-complete} to drop idempotent-completion.
\end{proof}

We conclude with a family of geometric examples including Krull-Schmidt cases, which provide us with novel instances of the Green correspondence:

\begin{Exa}[Equivariant sheaves] \label{Exa:sheaves}
As in \cite[Section~4.4]{BalmerDellAmbrogio18pp}, we may consider Mackey 2-functors for a \emph{fixed} group~$G_0$ arising from a category equipped with a $G_0$-action. For instance, we may look at the category of sheaves of modules over a locally ringed space~$X$ equipped with an action of the group~$G_0$. This yields a Mackey 2-functor $\MM$ for~$G_0$ whose value at a subgroup $G\leq G_0$ is the category $\MM(G)=\Mod(X/\!\!/G)$ of $G$-equivariant sheaves of $\mathcal O_X$-modules.
If $X$ is a noetherian scheme, we may also consider Mackey 2-functors with value at~$G$ the category $\Qcoh(X/\!\!/G)$ of quasi-coherent sheaves, or $\coh(X/\!\!/G)$ of coherent sheaves, or their (bounded) derived categories $\Der(\Qcoh(X/\!\!/G))$ or~$\Db(\coh(X/\!\!/G))$. If $X$ is a regular and proper variety over a field~$\kk$, then the latter Mackey 2-functor is also Krull-Schmidt. Indeed, like $\coh(X)$, these categories are $\kk$-linear and idempotent-complete, and are Hom-finite by \cite{Atiyah56} or~\cite[Corollaire~2.5]{SGA6}.

As easily verified by an explicit computation, all of these examples are cohomological in the sense of \Cref{Rem:cohomological}; in particular, the Green correspondence for the latter example should be considered for $\mathrm{char}(\kk)=p>0$.
\end{Exa}

\begin{proof}[Proof of \Cref{Cor:intro-geom}]
Use \Cref{Thm:Green-equivalence} for the Mackey 2-functor of \Cref{Exa:sheaves}, $\MM(-)=\Db(\coh(X/\!\!/-)$, and use \Cref{Rem:KS-quotients} to drop idempotent-completion.
\end{proof}

\begin{Rem}
\label{Rem:sheaves}%
The Mackey 2-functors of \Cref{Exa:sheaves} in the special case of $X=\mathrm{Spec}(\kk)$ equipped with the trivial $G_0$-action simply recover the ones of \Cref{Exa:Green-corr-modrep} (or more precisely their $G_0$-local version, \ie their restriction along the forgetful 2-functor $\gpdGzero \to \gpd$).
\end{Rem}

%------------------------------------------------------------------------------
%    Bibliography
%------------------------------------------------------------------------------

% uncomment the latter (and comment out the references below) to recompile the bibliography from articles.bib
%\bibliographystyle{alpha}
%\bibliography{articles}

%------------------------------------------------------------------------------
\end{document}